\documentclass[12pt, reqno]{amsart}
\usepackage[colorlinks,citecolor=blue]{hyperref}
\usepackage{amsmath}
\usepackage{amsthm}
\usepackage{amssymb}
\usepackage{graphicx}

\newcommand{\Hmm}[1]{\leavevmode{\marginpar{\tiny%
$\hbox to 0mm{\hspace*{-0.5mm}$\leftarrow$\hss}%
\vcenter{\vrule depth 0.1mm height 0.1mm width \the\marginparwidth}%
\hbox to 0mm{\hss$\rightarrow$\hspace*{-0.5mm}}$\\\relax\raggedright
#1}}}

\numberwithin{equation}{section}
\newtheorem{Thm}{Theorem}[section]
\newtheorem{Cor}{Corollary}[section]
\newtheorem{Lem}{Lemma}[section]
\newtheorem*{Claim}{Claim}
\newtheorem{Pro}{Proposition}[section]

\theoremstyle{definition}
\newtheorem{Def}{Definition}[section]
\theoremstyle{definition}
\newtheorem{Rmk}{Remark}[section]

\newcommand{\bel}[1]{\begin{equation}\label{#1}}

\newcommand{\be}{\begin{equation}}

\newcommand{\ba}{\begin{eqnarray}}
\newcommand{\ea}{\end{eqnarray}}

\newcommand{\qe}{\end{equation}}

\newtheorem{thesis}{Thesis}
\newcommand{\btl}[1]{\begin{thesis}\label{#1}}
\newcommand{\et}{\end{thesis}}

\theoremstyle{theorem}

\newtheorem{satz}{Proposition}[section]
\theoremstyle{corollary}

\theoremstyle{lemma}

\theoremstyle{definition}

\theoremstyle{proof}

\theoremstyle{remark}

\newcommand\blfootnote[1]{%
  \begingroup
  \renewcommand\thefootnote{}\footnote{#1}%
  \addtocounter{footnote}{-1}%
  \endgroup
}
\date{}
\begin{document}

\title[Cheeger constants for signed graphs]{Cheeger constants, structural balance, and spectral clustering analysis for signed graphs}
\author{Fatihcan M. Atay}
\address{Department of Mathematics, Bilkent University, 06800 Ankara, Turkey.}
\email{f.atay@bilkent.edu.tr}

\author{Shiping Liu}
\address{School of Mathematical Sciences, University of Science and Technology of China, 230026 Hefei, China.}
\email{spliu@ustc.edu.cn}
\blfootnote{2010 Mathematics Subject Classification 05C50, 05C22 (05C85, 39A12).}

\begin{abstract}
We introduce a family of multi-way Cheeger-type constants $\{h_k^{\sigma}, k=1,2,\ldots, n\}$ on a signed graph $\Gamma=(G,\sigma)$ such that $h_k^{\sigma}=0$ if and only if $\Gamma$ has $k$ balanced connected components. These constants are switching invariant and bring together in a unified viewpoint a number of important graph-theoretical concepts, including
the classical Cheeger constant, those measures of bipartiteness introduced by Desai-Rao, Trevisan, Bauer-Jost, respectively, on unsigned graphs,
and the frustration index (originally called the line index of balance by Harary) on signed graphs.
We further unify the (higher-order or improved) Cheeger and dual Cheeger inequalities for unsigned graphs as well as the underlying algorithmic proof techniques by establishing their corresponding versions on signed graphs. In particular, we develop a spectral clustering method for finding $k$ almost-balanced subgraphs, each defining a sparse cut. The proper metric for such a clustering is the metric on a real projective space.
We also prove estimates of the extremal eigenvalues of signed Laplace matrix in terms of number of signed triangles ($3$-cycles).
\end{abstract}

\maketitle

\tableofcontents
\section{Introduction}
Given a graph, there are two basic properties often of interest:
  (i) Is it \emph{connected}?
  (ii) Is it \emph{bipartite}?
We know from spectral graph theory that the two properties can be characterized by eigenvalues of the Laplace matrix of the graph.
The second smallest eigenvalue of the Laplace matrix vanishes if and only if the graph is disconnected (recall that the smallest eigenvalue is always zero). A quantitative characterization is given by the \emph{Cheeger inequality}:
the \emph{Cheeger constant} was introduced as a measure of connectedness, and gives upper and lower bounds of the second smallest eigenvalue \cite{Cheeger1970,Dodziuk1984,AM1985,Alon1986, Mohar88, Chung}. (In fact, such constants have been introduced in graph theory earlier by Poly\'{a} and Szeg\"{o} \cite{PS1951}, though without the connection to eigenvalues.) On the other hand, the bipartiteness property turns out to be closely related to the largest eigenvalue of the Laplace matrix. More precisely, the eigenvalues of the normalized Laplace matrix lie in the interval $[0,2]$, and the gap between $2$ and the largest eigenvalue vanishes if and only if the graph is bipartite. A quantitative version of this fact has been studied by several different but closely related approaches: Desai and Rao \cite{DR1994} introduced a \emph{non-bipartiteness parameter}, Trevisan \cite{Trevisan2012} introduced a \emph{bipartiteness ratio} and Bauer and Jost \cite{BJ} introduced a \emph{dual Cheeger constant}, as measures of bipartiteness. They also established the corresponding eigenvalue estimates, which we will refer to as \emph{dual Cheeger inequalities}, following the terminology of Bauer and Jost.

Although the connectedness and bipartiteness are two quite different properties, the proofs of Cheeger and dual Cheeger inequalities share quite a number of comparable ideas. The similarity of the spectral approaches of connectedness and bipartiteness is further observed in the proofs of higher order Cheeger inequalities of Lee, Oveis  Gharan and Trevisan \cite{LOT2013} and the higher order dual Cheeger inequalities of Liu \cite{Liu13}, and in the proofs of improved Cheeger inequalities and
bounds on bipartiteness ratio of Kwok, Lau, Lee, Oveis Gharan and Trevisan \cite{KLLOT2013}.

It is then natural to ask for the reason of such similarity.  In this paper, we present a systematic unified approach to the above-mentioned results in the framework of signed graphs. This clarifies the reason of the similarity and provides a deeper understanding of the relationship between Cheeger and dual Cheeger inequalities.

In fact, we  more generally study the relation between the spectra and the \emph{structural balance theory} of signed graphs, which is interesting in its own right.
Signed graphs and the idea of balance, introduced by Frank Harary \cite{Harary53} in 1953, and rediscovered since then in different contexts many times, are important models and tools for various research fields. The concepts were motivated and suggested by problems in social psychology \cite{Harary53,Harary57,CartHarary56} and have stimulated new methods for analyzing social networks \cite{KLB09,WYWLZ12,SM13}, biological networks \cite{SIMM13}, logical programming \cite{CostantiniProvetti10}, and so on.
Signed graphs also play important roles in various branches of mathematics, such as group theory, root systems (see \cite{CST} and the references therein), topology \cite{Cameron77,CameronWells86}, and even physics \cite{BMRU80}. By relating signed graphs with $2$-lifts of a graph, Bilu and Linial \cite{BL06} reduce the problem of constructing expander graphs to finding a signature with small spectral radius.
In a recent breakthrough, Marcus, Spielman, and Srivastava \cite{MSS, MSSICM} show the existence of infinite families of regular bipartite Ramanujan graphs of every degree larger than $2$, by proving a variant conjecture of Bilu and Linial about the existence of the signature of a given graph with very small spectral radius.

Mathematically, a signed graph $\Gamma=(G, \sigma)$ is an undirected graph $G=(V,E)$ with a signature $\sigma: E\rightarrow \{+1,-1\}$ on the edge set $E$.
One can think of the vertex set $V$ as a social group, where a positive (resp., negative) edge between two vertices indicates that the two members are friends (resp., enemies). The sign of a cycle in $G$ is defined as the product of the signs of all edges in it. A signed graph $\Gamma$ is called \emph{balanced} if all cycles in $G$ are positive. This is a crucial concept
due to Harary \cite{Harary53}, and
plays a central role in our unification of Cheeger and dual Cheeger inequalities in the framework of signed graphs.

The properties of being balanced can be characterized by the eigenvalues of the signed normalized Laplace matrix. More precisely, the smallest eigenvalue of the matrix vanishes if and only if the graph $\Gamma$ has a balanced connected component.
In this paper, we define a Cheeger-type constant $h_1^{\sigma}$ (see Definition \ref{def:signedCheeger}) based on Harary's balance theorem (Theorem \ref{thmHararyBalance}) such that
\begin{equation}\label{fact2}
\Gamma \text{ has a balanced connected component} \iff  h_1^{\sigma}=0.
\end{equation}
In the following, we will refer to this Cheeger-type constant of a signed graph as a \emph{signed Cheeger constant} for short. Similarly, we will also speak of signed inequalities and signed algorithms. We will then prove a signed Cheeger inequality (see Theorem \ref{thmCheegerEsti}), which is an estimate of the smallest eigenvalue from below and above in terms of $h_1^\sigma$. (For previous results in this aspect, see \cite{Belardo14,Hou05}.)

When the signature $\sigma$ on $\Gamma$ is \emph{switching equivalent} (see Definition \ref{def:switching}) to all-positive signature, having a balanced connected component is simply equivalent to the trivial property of having a connected component. When $\sigma$ is switching equivalent to all-negative signature, having a balanced connected component is equivalent to having a bipartite connected component.
Indeed, we show that the signed Cheeger constant $h_1^{\sigma}$ and its multi-way versions (Definition \ref{def:multiwayCheeger}) provide
a common extension of the classical Cheeger constant \cite{Cheeger1970,Dodziuk1984,AM1985,Alon1986, Mohar88, Chung},
 the non-bipartiteness parameter of Desai and Rao \cite{DR1994} (after a modification; see (\ref{eq:modDR})), the bipartiteness ratio of Trevisan \cite{Trevisan2012},
 and the dual Cheeger constant of Bauer and Jost \cite{BJ}.

The introduction of the signed Cheeger constant further enables us to develop a corresponding spectral clustering method on signed networks. We propose an algorithm for finding $k$ \emph{almost-balanced subgraphs} of a signed graph $\Gamma=(G,\sigma)$.
The novel point is that, after embedding the graph into the Euclidean space $\mathbb{R}^k$ via eigenfunctions, we find the proper metric for clustering points is a metric on the real projective space $P^{k-1}\mathbb{R}$ (see (\ref{eq:projmetric})) studied in \cite{Liu13}. Interestingly, this algorithm unifies the traditional spectral clustering for finding $k$ sparse cuts (e.g.~\cite{NJW2002,Luxburg07}) and the recent one for finding $k$ almost-bipartite subgraphs proposed in \cite{Liu13}.

We further explore the related theoretical analysis of this algorithm. Indeed, we unify the higher-order Cheeger \cite{LOT2013}, the higher-order dual Cheeger \cite{Liu13}, and the improved Cheeger inequalities, the improved bounds of bipartiteness ratio \cite{KLLOT2013} of unsigned graphs as higher order signed Cheeger inequalities (Theorem \ref{HigerCheeger}) and improved signed Cheeger inequalities (Theorem \ref{ImprChIntro}), in terms of our signed Cheeger constants.

Harary \cite{Harary57} defined a signed graph $\Gamma=(G,\sigma)$ to be \emph{antibalanced} if its negation $-\Gamma:=(G,-\sigma)$ is balanced.
Thus, $\Gamma$ is antibalanced if and only if every odd cycle in it is negative and every even cycle is positive. It is known that a connected signed graph is antibalanced if and only if  the largest eigenvalue of the signed normalized Laplace matrix equals $2$ (see \cite{LiLi09}). We obtain similar results concerning antibalance and the spectral gap between $2$ and the largest eigenvalue via an antithetical dual signed Cheeger constant (see (\ref{fact3})).


Finally, we prove estimates for the smallest and largest eigenvalues  of the signed normalized Laplace matrix in terms of signed $3$-cycles (we will speak of signed triangles in the following). By definition, the presence of positive (resp., negative) triangles implies that $\Gamma$ cannot be antibalanced (resp., balanced). Therefore, the number of signed triangles relate naturally to the spectral gaps between $0$ and the smallest eigenvalue, and, between $2$ and the largest eigenvalue. We also establish an upper bound estimate for the largest eigenvalue of the signed non-normalized Laplace matrix, in terms of vertex degrees and number of positive triangles. This result
improves the estimate by Hou, Li, and Pan \cite{HouLiPan03}, which is only in terms of vertex degrees.
In fact, our estimate answers the question asked in their paper \cite[remark after Theorem 3.5]{HouLiPan03}.

\section{Preliminaries}
Let $\Gamma=(G,\sigma)$ be a signed graph, where $G=(V,E)$ is an undirected graph and $\sigma: E\to \{+1,-1\}$ is a signature function on the edges. We collect some basic results from structural balance theory and spectral theory of signed graphs in this section.

We first introduce some notation. We say $u,v \in V$ are neighbors when $e=\{u,v\}\in E$, and write $u\sim v$. For ease of notation we write $\sigma(uv):=\sigma(\{u,v\})$ for the sign of an edge.
In addition to the sign, we also assign a positive symmetric weight $w_{uv}$ to every edge $e=\{u,v\}\in E$, and set $w_{uv}=0$ if $e=\{u,v\}\notin E$.
We say the graph is unweighted if $w_{uv}\equiv 1,\,\,\forall\,\, \{u,v\}\in E$. The degree $d_u$ of a vertex $u$ is defined as $d_u=\sum_{v\in V}w_{uv}$.
We will restrict ourselves to signed simple graphs, i.e., the case when the underlying graph $G$ has no self-loops or multi-edges.
We also consider a general positive measure $\mu: V\rightarrow \mathbb{R}_{>0}$ on the vertex set.

\subsection{Harary's balance theorem and bipartition}
Recall that $\Gamma$ is called balanced if all cycles in $G$ is positive. The following structure theorem for balance was proved in \cite{Harary53}.

\begin{Thm}[Harary's Balance Theorem]\label{thmHararyBalance}
 A signed graph $\Gamma$ is balanced if and only if there exists a bipartition of the entire set $V$ into two disjoint subsets $V_1$ and $V_2$ (one of which may be empty) such that each positive edge connects two vertices of the same subset and each negative edge connects two vertices of different subsets.
\end{Thm}

By reversing the signature, Harary gave the antithetical dual result for antibalance. Recall that $\Gamma$ is called antibalanced if every odd cycle is negative and every even cycle is positive.

\begin{Thm}\cite{Harary57} \label{thmAntibalance}
 A signed graph $\Gamma$ is antibalanced if and only if  there exists a bipartition of the entire set $V$ into two disjoint subsets $V_1$ and $V_2$ (one of which may be empty), such that each negative edge connects two vertices of the same subset and each positive edge connects two vertices of different subsets.
\end{Thm}

\subsection{Switching equivalence} In this subsection, we discuss an important operation of signatures on $G$, called switching.

\begin{Def}[Switching]\label{def:switching} The operation of reversing the signs of all edges connecting a subset $S\subseteq V$ and its complement is called  \emph{switching the subset $S$}. Two signatures $\sigma: E\to \{+1,-1\}$ and $\sigma': E\to \{+1,-1\}$ are said to be \emph{switching equivalent} if there exists a subset $S\subseteq V$ such that $\sigma'$ can be obtained from $\sigma$ by switching the subset $S$. We write $\sigma'\approx \sigma$ in this case.
\end{Def}

Switching equivalence is an equivalence relation on signatures of a fixed underlying graph.
We call the corresponding equivalent classes the \emph{switching classes}, and denote the switching class of $\sigma$ by $[\sigma]$.

By definition, the sign of any cycle is preserved by the switching operation. Therefore, we have the following basic property.

\begin{Pro}
Being balanced is a switching invariant property.
\end{Pro}

There is a different but equivalent way to describe switching operations, which is more convenient for our later discussions. A function $\theta: V\rightarrow \{+1, -1\}$ is called a \emph{switching function}.
Switching the signature of $\Gamma=(G, \sigma)$ by $\theta$ refers to the operation of changing $\sigma$ to $\sigma^{\theta}$ via
\begin{equation*}
\sigma^{\theta}(uv):=\theta(u)\sigma(uv)\theta(v), \,\,\forall\,\, \{u,v\}\in E.
\end{equation*}
Equivalently speaking, switching $\sigma$ by $\theta$ means reversing the signs of all edges between the set $V_{\theta}^-:=\{u\in V: \theta(u)=-1\}$ and its complement.
Therefore, this operation is just switching the subset $V_{\theta}^-$ of $V$.
Given $v\in V$, define $\theta_v(u)=-1$ if $u=v$ and $+1$ otherwise.
A vertex switching at $v$, i.e. switching the vertex $v$, means switching $\sigma$ by $\theta_v$.
Note that $\theta=\prod_{v\in V_{\theta}^-}\theta_v$; thus,
switching a subset of $V$ is equivalent to switching every vertex in it one after another.

Zaslavsky \cite{Zaslavsky82} proved the following useful characterization.

\begin{Thm}[Zaslavsky's switching lemma] A signed graph $\Gamma=(G,\sigma)$ is balanced if and only if $\sigma$ is switching equivalent to the all-positive signature, and it is antibalanced if and only if $\sigma$ is switching equivalent to the all-negative signature.
\end{Thm}

For more details and history about switching, we refer to \cite{ZaslavskyMatrices} and the references therein.

\subsection{Laplace matrices and basic spectral theory} In this section, we discuss the spectra of signed graphs, which are in fact switching invariant (see, e.g., \cite{ZaslavskyMatrices}). Let $A^\sigma$ be the signed adjacency matrix. That is, we have for any $u,v\in V$,
$$(A^\sigma)_{uv}=\left\{
                    \begin{array}{ll}
                      \sigma(uv), & \hbox{if $u\sim v$;} \\
                      0, & \hbox{otherwise.}
                    \end{array}
                  \right.
$$

Given a switching function $\theta: V\to \{+1,-1\}$, we denote by $D(\theta)$ the diagonal matrix with $D(\theta)_{uu}=\theta(u), \,\forall \, u\in V$. It is then straightforward to check that
 \begin{equation}
 A^{\sigma^{\theta}}=D(\theta)^{-1}A^{\sigma}D(\theta).
 \end{equation}
Therefore, the spectrum of the matrix $A^\sigma$ is switching invariant.

\begin{Def}[signed Laplace matrices] Let $D$ be the diagonal degree matrix. We call $L^\sigma:=D-A^\sigma$ the \emph{signed non-normalized Laplace matrix}, and $\Delta^\sigma:=I-D^{-1}A^\sigma$ the \emph{signed normalized Laplace matrix}, where $I$ denotes the $|V|\times |V|$ identity matrix.
\end{Def}

\begin{Rmk}
In the literature, the normalized Laplace matrix of an  unsigned graph is defined either as the matrix $I-D^{-\frac{1}{2}}AD^{-\frac{1}{2}}$ \cite{Chung} or $I-D^{-1}A$ \cite{BJ}, where $A$ is the unsigned adjacency matrix. The latter one is also called the random walk Laplacian. The two matrices are unitarily equivalent and hence share the same spectrum. The signed normalized Laplace matrix $\Delta^\sigma$ we use here is an extension of the latter matrix. it has the same spectrum as the matrix $I-D^{-\frac{1}{2}}A^\sigma D^{-\frac{1}{2}}$. The matrix $\Delta^{\sigma}$ also appears naturally in the context of graph drawing and electrical networks \cite{KSLLDA}.
\end{Rmk}

The following property follows directly from our discussion of the spectrum of $A^\sigma$.
 \begin{Pro}[\cite{ZaslavskyMatrices}]\label{satz:spectrumswitchinginvariant}
 The spectrum of $\Delta^{\sigma}$ or $L^{\sigma}$ for a signed graph $\Gamma=(G, \sigma)$ is switching invariant.
 \end{Pro}

It is well-known that the eigenvalues of $\Delta^{\sigma}$ lie in the interval $[0,2]$. We can list them (counting with multiplicity) as
\begin{equation*}
0\leq \lambda_1(\Delta^{\sigma})\leq \lambda_2(\Delta^{\sigma})\leq \cdots \leq \lambda_n(\Delta^{\sigma})\leq 2,
\end{equation*}
where $n$ is the cardinality of $V$. Moreover, the cases when $0$ and $2$ are eigenvalues are characterized as below; see, e.g., \cite{Zaslavsky82,HouLiPan03,LiLi09}.
\begin{align}
\Gamma \text{ has a balanced connected component } &\Leftrightarrow\,\, \lambda_1(\Delta^{\sigma})=0,\label{fact1}\\
\Gamma \text{ has an antibalanced connected component } &\Leftrightarrow\,\, \lambda_n(\Delta^{\sigma})=2.
\end{align}

Let $\mu_d$ denote the degree measure on $V$, i.e. $\mu_d(u)=d_u$, $\forall u\in V$.
The operator form of $\Delta^{\sigma}$ can be expressed by its action on any function $f: V\rightarrow \mathbb{R}$ and any $u\in V$ as
\begin{equation}\label{LaplaceDef}
\Delta^{\sigma} f(u)=\frac{1}{\mu_d(u)}\sum_{v, v\sim u}w_{uv}(f(u)-\sigma(uv)f(v)).
\end{equation}
Replacing $\mu_d$ above by the constant measure $\mu_1\equiv 1$ yields the operator form for $L^{\sigma}$.
For a general measure $\mu$, we denote the corresponding inner product of two functions $f,g: V\rightarrow \mathbb{R}$ by
$$(f,g)_{\mu}=\sum_{u\in V}\mu(u)f(u)g(u).$$
The signed Rayleigh quotient of a map $\Phi: V\rightarrow \mathbb{R}^k$ is given by
\begin{equation}
\mathcal{R}^{\sigma}(\Phi)=\frac{\sum_{u\sim v}w_{uv}\Vert\Phi(u)-\sigma(uv)\Phi(v)\Vert^2}{\sum_{u\in V}\mu(u)\Vert\Phi(u)\Vert^2}.
\end{equation}
We also define a dual version of the Rayleigh quotient of $\Phi$ by
\begin{equation}
\widetilde{\mathcal{R}}^{\sigma}(\Phi)=\frac{\sum_{u\sim v}w_{uv}\Vert\Phi(u)+\sigma(uv)\Phi(v)\Vert^2}{\sum_{u\in V}\mu(u)\Vert\Phi(u)\Vert^2}.
\end{equation}
The Courant-Fisher-Weyl min-max principle says that the $k$-th eigenvalue $\lambda_k$ of $\Delta^{\sigma}$ (or $L^{\sigma}$) satisfies
\begin{equation}\label{keigenvalue}
\lambda_k=\min_{\substack{f_1,f_2,\ldots, f_k\not\equiv 0\\(f_i,f_j)_{\mu}=0, \forall i\neq j}} \;\; \max_{\substack{f\not\equiv 0\\f\in \text{span}\{f_1, f_2,\ldots, f_k\}}}
\mathcal{R}^{\sigma}(f).
\end{equation}
In particular, we have
\begin{equation}\label{varForleatLar}
\lambda_1(\Delta^{\sigma})=\min_{f\not\equiv 0}\mathcal{R}^{\sigma}(f), \text{ and } 2-\lambda_N(\Delta^{\sigma})=\min_{f\not\equiv 0}\widetilde{\mathcal{R}}^{\sigma}(f).
\end{equation}

\begin{Lem}\label{lemma:dual}
For any $1\leq k\leq n$, it holds that $$2-\lambda_{N-k+1}(\Delta^{\sigma})=\lambda_k(\Delta^{-\sigma}).$$
\end{Lem}

This follows immediately from the fact that $\widetilde{\mathcal{R}}^{\sigma}(f)=\mathcal{R}^{-\sigma}(f)$.
The support of a map $\Phi:V\to\mathbb{R}$ is defined as
\begin{equation*}
\text{supp}(\Phi):=\{u\in V: \Phi(u)\neq 0\}.
\end{equation*}
By (\ref{keigenvalue}), one can derive the following lemma (see, e.g., \cite{KLLOT2013}).

\begin{Lem}\label{lemmaPreliminary}
For any $k$ disjointly supported functions $f_1, f_2, \ldots, f_k: V\rightarrow \mathbb{R}$,
\begin{equation}
\lambda_k\leq 2\max_{1\leq i\leq k}\mathcal{R}^{\sigma}(f_i).
\end{equation}
\end{Lem}


We refer to \cite{Hou05,LiLi08,GHZ11,AT14,Belardo14,Reff} for more results in the spectral theory of signed graphs.

\section{(Multi-way) Signed Cheeger constants}


In this section, we define the signed Cheeger constant $h_1^\sigma$ mentioned in the introduction and prove that $h_1^\sigma$ is switching invariant. We relate $h_1^\sigma$ to several existing graph-theoretic concepts. We also define the corresponding multi-way signed Cheeger constants.

\subsection{Definition of the signed Cheeger constant}
We first define some notation.
Given two subsets  $V_1, V_2$ of $V$, denote the set of edges lying between $V_1$ and $V_2$ by
$$E(V_1,V_2):=\{\{u,v\}\in E: u\in V_1, v\in V_2\}.$$
We distinguish the set of positive edges and the set of negative edges by
$$E^+(V_1,V_2):=\{\{u,v\}\in E: u\in V_1, v\in V_2, \sigma(uv)=+1\},$$
and
$$E^-(V_1,V_2):=\{\{u,v\}\in E: u\in V_1, v\in V_2, \sigma(uv)=-1\}.$$
We further define their weighted cardinality as
\begin{align*}
&|E(V_1, V_2)|=\sum_{u\in V_1}\; \sum_{v\in V_2, \{u,v\}\in E(V_1,V_2)}w_{uv}, \\
&|E^+(V_1, V_2)|=\sum_{u\in V_1}\; \sum_{v\in V_2, \{u,v\}\in E^+(V_1,V_2)}w_{uv},\\
&|E^-(V_1, V_2)|=\sum_{u\in V_1}\; \sum_{v\in V_2, \{u,v\}\in E^-(V_1,V_2)}w_{uv}.
\end{align*}
When $V_1=V_2$, we write $$|E(V_1)|:=|E(V_1,V_1)|$$ and $$|E^+(V_1)|:=|E^+(V_1,V_1)|, \,\,\,|E^-(V_1)|:=|E^-(V_1,V_1)|$$ for short.
Keep in mind that, in this case, the weighted cardinality defined above counts every edge weight twice. We adopt this definition of weighted cardinality for the convenience of later discussions.

For a subset $S\subseteq V$, we define its volume as $$\mathrm{vol}_{\mu}(S)=\sum_{u\in S}\mu(u).$$


Let $(V_1, V_2)$ denote a \emph{sub-bipartition} of $V$, that is, $V_1$ and $V_2$ are subsets of $V$ satisfying $\emptyset\neq V_1\cup V_2\subseteq V$ and $V_1\cap V_2=\emptyset$.
We define the \emph{signed bipartiteness ratio} of $(V_1, V_2)$ as
\begin{equation}
\beta^{\sigma}(V_1, V_2)=\frac{2|E^{+}(V_1, V_2)|+|E^-(V_1)|+|E^-(V_2)|+|E(V_1\cup V_2, \overline{V_1\cup V_2})|}{\mathrm{vol}_{\mu}(V_1\cup V_2)},
\end{equation}
where $\overline{V_1\cup V_2}$ is the complement of $V_1\cup V_2$ in $V$.

Observe that $\beta^\sigma(V_1,V_2)$ can be considered as a modification of the expansion
$$\frac{|E(V_1\cup V_2, \overline{V_1\cup V_2})|}{\mathrm{vol}_{\mu}(V_1\cup V_2)}$$
of the set $V_1\cup V_2$. If the additional quantity in the numerator
$$2|E^{+}(V_1, V_2)|+|E^-(V_1)|+|E^-(V_2)|$$
 vanishes, then, by Harary's balanced theorem (Theorem \ref{thmHararyBalance}), we conclude that the induced signed subgraph of $V_1\cup V_2$ in $\Gamma=(G,\sigma)$ is balanced. Moreover, $V_1$ and $V_2$ give a partition of the induced subgraph satisfying the properties described in Theorem \ref{thmHararyBalance}.

%
%

\begin{Def}[Signed Cheeger constant]\label{def:signedCheeger}
For a signed graph $\Gamma=(G, \sigma)$,  \emph{the signed Cheeger constant} $h_1^{\sigma}(\mu)$ is defined as
\begin{equation}\label{defSCheeger}
 h_1^{\sigma}(\mu)=\min_{(V_1, V_2)}\beta^{\sigma}(V_1, V_2),
\end{equation}
 where the minimum is taken over all possible sub-bipartitions of $V$.
\end{Def}

With this definition, we observe directly that $h_1^\sigma(\mu)=0$ if and only if there exists a sub-bipartition $(V_1,V_2)$ of $V$ with $\beta^\sigma(V_1,V_2)=0$. From the above discussion about the signed bipartiteness ratio, we have that $h_1^\sigma(\mu)=0$ if and only if $\Gamma$ has a balanced connected component. That is, the statement (\ref{fact2}) mentioned in the introduction holds.

\subsection{Switching invariant property}
We prove that $h_1^{\sigma}(\mu)$ is switching invariant.

\begin{satz}\label{prop:switching invariant}
Let $\Gamma=(G,\sigma)$ be a signed graph. For any switching function $\theta:V\rightarrow \{+1,-1\}$,
\begin{equation}
h^{\sigma}_1(\mu)=h^{\sigma^{\theta}}_1(\mu).
\end{equation}
\end{satz}

This property is a direct corollary of the following lemma.
\begin{Lem}\label{lemma:switchMove}
For any switching function $\theta:V\rightarrow \{+1,-1\}$ and any sub-bipartition $(V_1, V_2)$, there exists a sub-bipartition $(V_1', V_2')$, such that $V_1'\cup V_2'=V_1\cup V_2$, and
\begin{equation}\label{betalemma}
\beta^{\sigma^{\theta}}(V_1', V_2')=\beta^{\sigma}(V_1, V_2).
\end{equation}
Moreover, when $\sigma^{\theta}$ can be obtained from $\sigma$ via switching a subset $S\subseteq~V_1$, we have
\begin{equation}\label{betalemma2}
\beta^{\sigma^{\theta}}(V_1\setminus S, V_2\cup S)=\beta^{\sigma}(V_1, V_2).
\end{equation}
\end{Lem}

\begin{proof}
We only need to prove the lemma for a vertex switching at $u\in V$, that is, a switching of $\sigma$ by $\theta_u$. (Recall that $\theta_u(v)=-1$
if $v=u$ and $+1$ otherwise, and $\sigma^{\theta_u}(uv)=-\sigma(uv)$ for all $v\sim u$.)

Now, if $u\in \overline{V_1\cup V_2}$, the vertex switching at $u$ does not change the signed bipartiteness ratio; hence $V_1':=V_1$ and $V_2':=V_2$
satisfy the equality (\ref{betalemma}).

If $u\in V_1\cup V_2$, suppose without loss of generality that $u\in V_1$. After the vertex switching at $u$, we have
\begin{align*}
&(\beta^{\sigma^{\theta_u}}(V_1, V_2)-\beta^{\sigma}(V_1, V_2))\,\mathrm{vol}_{\mu}(V_1\cup V_2)\\
&=2\sum_{\substack{v\in V_2\\ \sigma(uv)=-1}}w_{uv}-2\sum_{\substack{v\in V_2\\ \sigma(uv)=+1}}w_{uv}+2\sum_{\substack{v\in V_1\\ \sigma(uv)=+1}}w_{uv}-2\sum_{\substack{v\in V_1\\ \sigma(uv)=-1}}w_{uv}.
\end{align*}
In this case, we set $V_1':=V_1\setminus \{u\}$ and $V_2':=V_2\cup \{u\}$, that is, we move the vertex $u$ from $V_1$ into $V_2$. Then we observe that
\begin{align*}
&(\beta^{\sigma^{\theta_u}}(V_1', V_2')-\beta^{\sigma^{\theta_u}}(V_1, V_2))\mathrm{vol}_{\mu}(V_1'\cup V_2')\\
&=-2\sum_{\substack{v\in V_1\\ \sigma^{\theta_u}(uv)=-1}}w_{uv}+2\sum_{\substack{v\in V_1\\ \sigma^{\theta_u}(uv)=+1}}w_{uv}-2\sum_{\substack{v\in V_2\\ \sigma^{\theta_u}(uv)=+1}}w_{uv}+2\sum_{\substack{v\in V_2\\ \sigma^{\theta_u}(uv)=-1}}w_{uv}.
\end{align*}
Adding the two equalities above, we arrive at (\ref{betalemma}).
\end{proof}

\subsection{Relations to other graph-theoretic concepts}
We use Proposition \ref{prop:switching invariant} to explain the relations of the signed Cheeger constant to several other graph-theoretic concepts.

We denote by $\sigma_+$ (resp., $\sigma_-$) the all-positive (resp., all-negative) signature. We consider the signed Cheeger constant of two particular classes of signatures: (i) $\sigma\in [\sigma_-]$, and (ii) $\sigma\in [\sigma_+]$.

Recall that the \emph{bipartiteness ratio }of Trevisan \cite{Trevisan2012} is defined as
\begin{align*}
\beta:=\min_{(V_1, V_2)}\frac{2|E(V_1, V_2)|+|E(V_1)|+|E(V_2)|+|E(V_1\cup V_2, \overline{V_1\cup V_2})|}{\mathrm{vol}_{\mu}(V_1\cup V_2)}.
\end{align*}
Furthermore, Bauer and Jost \cite{BJ} introduced the \emph{dual Cheeger constant}
$$\overline{h}:=\max_{(V_1,V_2)}\frac{2|E(V_1,V_2)|}{\mathrm{vol}_{\mu}(V_1\cup V_2)}.$$
These two concepts are closely related, because, due to the fact that
$$\mathrm{vol}_{\mu}(V_1\cup V_2)=|E(V_1)|+|E(V_2)|+|E(V_1\cup V_2, \overline{V_1\cup V_2})|,$$
one has
$$\beta=1-\overline{h}.$$
Observe that, by definition,
$$h_1^{\sigma_-}(\mu)=\beta=1-\overline{h}.$$
Therefore, we obtain the following fact by Proposition \ref{prop:switching invariant}:
If $\sigma\in[\sigma_-]$, that is, if $\Gamma=(G, \sigma)$ is antibalanced, then $h_1^{\sigma}(\mu)=\beta=1-\overline{h}$.

We proceed to consider $h_1^\sigma(\mu)$ when $\sigma\in [\sigma_+]$. To this end, we first prepare the following reformulation of the signed Cheeger constant. Recall that the expansion (or conductance) of a subset $S\subseteq V$ is defined as
\begin{equation}
 \rho(S):=\frac{|E(S, \overline{S})|}{\mathrm{vol}_{\mu}(S)}.
\end{equation}
We define a \emph{signed expansion} of $S\subseteq V$ in $\Gamma$ to be
\begin{equation}
 \rho^{\sigma}(S):=\frac{|E^-(S)|+|E(S, \overline{S})|}{\mathrm{vol}_{\mu}(S)}.
\end{equation}
We have the following relation between $h_1^{\sigma}(\mu)$ and signed expansions.

\begin{Cor}\label{cor:signed expansion}
Let $\Gamma=(G,\sigma)$ be a signed graph. Then,
\begin{equation}\label{eq:signed expansion}
h^{\sigma}_1(\mu)=\min_{\sigma'\in [\sigma]}\; \min_{\emptyset\neq S\subseteq V}\rho^{\sigma'}(S).
\end{equation}
\end{Cor}

\begin{proof}
We denote by $(V_1, V_2)_S$ a bipartition of $S$, i.e., $V_1\cup V_2=S$, $V_1\cap V_2=\emptyset$.
We claim that
\begin{equation}\label{eq:bipartANDexpansion}
\min_{(V_1, V_2)_S}\beta^{\sigma}(V_1, V_2)=\min_{\sigma'\in [\sigma]}\rho^{\sigma'}(S).
\end{equation}
Let $V_1^0$, $V_2^0$ be the bipartition of $S$ that achieves the minimum on the left hand side of (\ref{eq:bipartANDexpansion}). Suppose that $\sigma_0$ can be obtained from $\sigma$ via switching the subset $V_1^0$.
Then by Lemma \ref{lemma:switchMove},
\begin{equation}
\beta^{\sigma}(V_1^0, V_2^0)=\beta^{\sigma_0}(S, \emptyset)=\rho^{\sigma_0}(S)\geq \min_{\sigma'\in [\sigma]}\rho^{\sigma'}(S).
\end{equation}
Moreover, the inequality above can only be an equality, because otherwise there would exist a $\sigma'\in [\sigma]$ such that
\begin{equation*}
\beta^{\sigma'}(S, \emptyset)=\rho^{\sigma'}(S)< \beta^{\sigma}(V_1^0, V_2^0).
\end{equation*}
By Lemma \ref{lemma:switchMove}, we could then find a bipartition $V_1'$, $V_2'$ of $S$, such that
\begin{equation*}
\beta^{\sigma'}(S, \emptyset)=\beta^{\sigma}(V_1', V_2')< \beta^{\sigma}(V_1^0, V_2^0),
\end{equation*}
which is a contradiction.
Hence (\ref{eq:bipartANDexpansion}) holds. Then (\ref{eq:signed expansion}) follows directly.
\end{proof}
Using Corollary \ref{cor:signed expansion}, we observe the following: When $\sigma\in [\sigma_+]$, that is, when $\Gamma=(G,\sigma)$ is balanced, we have $$h_1^{\sigma}(\mu)=\min_{\emptyset\neq S\subseteq V}\rho^{\sigma_+}(S)=\min_{\emptyset\neq S\subseteq V}\frac{|E(S,\overline{S})|}{\mathrm{vol}_\mu(S)}.$$
This is the \emph{one-way Cheeger constant} \cite{Miclo2008,LOT2013}, which trivially vanishes.

Next, we compare the signed Cheeger constant with a constant due to Desai and Rao \cite{DR1994}. On unweighted graphs, Desai and Rao \cite{DR1994} introduced the \emph{non-bipartiteness parameter}
\begin{equation}\label{eq:DesaiRao}
 \alpha:=\min_{\emptyset\neq S\subseteq V}\frac{e_{\min}(S)+|E(S,\overline{S})|}{\mathrm{vol}_{\mu}(S)},
\end{equation}
where $e_{\min}(S)$ is the minimum number of edges that need to be removed from the induced subgraph of $S$ in order to make it bipartite. Hou \cite{Hou05} extended this notion to a signed graph $\Gamma=(G,\sigma)$ as
\begin{equation}\label{eq:Hou}
 \alpha^{\sigma}:=\min_{\emptyset\neq S\subseteq V}\frac{e^{\sigma}_{\min}(S)+|E(S,\overline{S})|}{\mathrm{vol}_{\mu}(S)},
\end{equation}
where $e^{\sigma}_{\min}(S)$ is the minimum number of edges that need to be removed from the induced subgraph of $S$ in order to make it balanced.
By definition, $(G, \sigma_-)$ is balanced if and only if $G$ has no odd cycles, i.e., $G$ is bipartite. Therefore, $\alpha^{\sigma_-}=\alpha$.

We notice that the quantity $e_{\min}^{\sigma}(V)$ coincides with the \emph{line index of balance} of $\Gamma$ introduced by Harary \cite{Harary59} (see also \cite{AR58}), which was later called the \emph{frustration index} (suggested to Harary by Zaslavsky in private communication from the work of Toulouse \cite{VT77}) and studied extensively, e.g.~\cite{HararyKabell80,AACE81,SZ94,Belardo14}. For a recent empirical approach of this index, see \cite{FIA11}.

For a subset $S$, we define a weighted version of the index $e^{\sigma}_{\min}(S)$ as follows. Let $\Gamma_S$ denote the induced signed graph of $S$, and $E_{|\Gamma_S}$ denote the edge set of $\Gamma_S$. Define
\begin{equation}\label{eq:weightedfrustration}
\begin{split}
& e^{\sigma}_{\min}(S):= \\
& \min_{E_1\subseteq E_{|\Gamma_S}}\{\sum_{\{u,v\}\in E_1}w_{uv}\mid \Gamma_S \text{ becomes balanced after deleting edges in } E_1\}.
\end{split}
\end{equation}
We modify the above notions (\ref{eq:DesaiRao}) and (\ref{eq:Hou}) to give the following parameter associated to $S\subseteq V$:
\begin{equation}
 \overline{\alpha}^{\sigma}(S):=\frac{2e^{\sigma}_{\min}(S)+|E(S,\overline{S})|}{\mathrm{vol}_{\mu}(S)}.
\end{equation}
This leads to the following reformulation of the signed Cheeger constant.

\begin{Cor}\label{cor:DesaiRaoHoumodified}
Let $\Gamma=(G,\sigma)$ be a signed graph. Then
\begin{equation}\label{eq:modDR}
 h_1^{\sigma}(\mu)=\min_{\emptyset\neq S\subseteq V}\overline{\alpha}^{\sigma}(S).
\end{equation}
\end{Cor}

Thus, the constant $h_1^{\sigma}(\mu)$ is larger than the parameter $\alpha^{\sigma}$ of Hou \cite{Hou05} in general. Corollary \ref{cor:DesaiRaoHoumodified} follows directly from the next lemma.

\begin{Lem}\label{lemma:threeequal}
For any $\emptyset\neq S\subseteq V$, we have
\begin{equation}\label{eq:threeequal}
\min_{(V_1, V_2)_S}\beta^{\sigma}(V_1, V_2)=\min_{\sigma'\in [\sigma]}\rho^{\sigma'}(S)=\overline{\alpha}^{\sigma}(S).
\end{equation}
\end{Lem}

\begin{proof}
The first equality follows from (\ref{eq:bipartANDexpansion}).
To prove the second equality,
let $\Gamma_S$ denote the induced signed graph of $S$ and let $\sigma_0$ be the signature that achieves $\min_{\sigma'\in [\sigma]}\rho^{\sigma'}(S)$.
It is easy to see that
\begin{equation*}
 2e_{\min}^{\sigma}(S)\leq |E^-(S)|(\sigma_0).
\end{equation*}
Therefore, $\overline{\alpha}^{\sigma}(S)\leq \min_{\sigma'\in [\sigma]}\rho^{\sigma'}(S)$.

Now let $\Gamma_S'$ be the balanced graph obtained from $\Gamma_S$ by deleting edges in $E_1$, which is the subset of $E_{|\Gamma_S}$ that attains the minimum in the definition (\ref{eq:weightedfrustration}) of $e_{\min}^{\sigma}(S)$. By Theorem \ref{thmHararyBalance}, there exists a bipartition $V_1$, $V_2$ of $S$ such that $$|E^+_{|\Gamma_S'}(V_1, V_2)|=|E^-_{|\Gamma_S'}(V_1)|=|E^-_{|\Gamma_S'}(V_2)|=0.$$
This implies
\begin{equation*}
 2e_{\min}^{\sigma}(S)=2|E^+_{|\Gamma_S}(V_1, V_2)|+|E^-_{|\Gamma_S}(V_1)|+|E^-_{|\Gamma_S}(V_2)|.
\end{equation*}
Hence $\overline{\alpha}^{\sigma}(S)\geq \min_{(V_1, V_2)_S}\beta^{\sigma}(V_1, V_2)$. This proves the second equality.
\end{proof}
\begin{Rmk}
 The (unweighted) equality $2e_{\min}^{\sigma}(S)=\min_{\sigma'\in [\sigma]}|E^-(S)|(\sigma')$ seems to be folklore among graph theorists; see \cite[Theorem 3.3]{Zaslavsky10}.
We include a proof here for completeness.
\end{Rmk}



We compare the signed Cheeger constant with the \emph{degree of balance }$b(\Gamma)$ of a signed graph $\Gamma$ introduced by Cartwright and Harary \cite{CartHarary56}. In \cite{CartHarary56}, they also aimed at quantifying the deviation of a signed graph from being balanced. Their constant $b(\Gamma)$ is defined as
\begin{equation}
 b(\Gamma):=\frac{\text{the number of positive cycles of } \Gamma}{\text{the number of cycles of } \Gamma}.
\end{equation}
Observe $b(\Gamma)\in [0,1]$.
Smaller values of  $1-b(\Gamma)$ imply that $\Gamma$ is closer to being balanced.
Consider the signed graph $\Gamma=(\mathcal{C}_n, \sigma)$ where $\mathcal{C}_n$ is the unweighted cycle graph on $n$ vertices and $\sigma$ is the signature such that $\Gamma$ has exactly one negative edge. Intuitively, $\Gamma$ is close to being balanced. Actually, we have
\begin{equation}
 1-b(\Gamma)=1 \text{  and  }h^{\sigma}_1(\mu_d)=\frac{1}{n}.
\end{equation}
This shows that the constant $h^{\sigma}_1(\mu_d)$ is finer than $b(\Gamma)$.

\subsection{Multi-way signed Cheeger constants}
Extending Definition \ref{defSCheeger} in the spirit of \cite{Miclo2008,LOT2013,Liu13}, we can define a family of multi-way signed Cheeger constant $\{h_k^{\sigma}(\mu), k=1,2,\ldots n\}$ in a natural way.

\begin{Def}\label{def:multiwayCheeger}
For $1\leq k\leq n$, the $k$-way signed Cheeger constant $h^{\sigma}_k(\mu)$ of a signed graph $\Gamma=(G,\sigma)$ is defined as
\begin{equation}\label{eq:signedCheeger}
h^{\sigma}_k(\mu):=\min_{\{(V_{2i-1}, V_{2i})\}_{i=1}^{k}}\; \max_{1\leq i\leq k}\beta^{\sigma}(V_{2i-1}, V_{2i}),
\end{equation}
where the minimum is taken over the space of all possible $k$ pairs of disjoint sub-bipartitions $(V_1, V_2),(V_3, V_4),\ldots, (V_{2k-1}, V_{2k})$. To ease the notation, we denote this space by $\mathrm{Pair}(k)$ and call every element of $\mathrm{Pair}(k)$ a $k$-sub-bipartition of $V$.
\end{Def}

The quantity $h_1^{\sigma}(\mu)$ defined in (\ref{defSCheeger}) is the first one of this family.  We have the monotonicity $h^{\sigma}_k(\mu)\leq h^{\sigma}_{k+1}(\mu)$.
Moreover, Theorem \ref{thmHararyBalance} implies the following property.
\begin{satz}
For a signed graph $\Gamma=(G,\sigma)$ and $1\leq k\leq n$, it holds that
   $h^{\sigma}_k(\mu)=0$ if and only if $\Gamma$ has $k$ balanced connected components.
\end{satz}

Roughly speaking, the $k$-way signed Cheeger constant is a ``mixture" of the $k$-way Cheeger constant $h_k(\mu)$ introduced by Miclo \cite{Miclo2008} (see also \cite{LOT2013}) and the $k$-way dual Cheeger constant $\overline{h}_k(\mu)$ in \cite{Liu13} for unsigned graphs.

By Lemma \ref{lemma:switchMove}, $h_k^{\sigma}(\mu)$ is switching invariant for any $1\leq k\leq n$. Furthermore, Lemma \ref{lemma:threeequal} implies the following equivalent expressions for $h_k^{\sigma}(\mu)$:
\begin{align*}
 h_k^{\sigma}(\mu)=&\min_{\sigma'\in [\sigma]}\;\min_{\{S_i\}_{i=1}^k}\;\max_{1\leq i\leq k}\rho^{\sigma'}(S_i)\\
=&\min_{\{S_i\}_{i=1}^k}\;\max_{1\leq i\leq k}\overline{\alpha}^{\sigma}(S_i),
\end{align*}
where the minimum $\min_{\{S_i\}_{i=1}^k}$ is taken over the space of all possible $k$-subpartitions, $S_1, S_2, \ldots, S_k$, where $S_i\neq \emptyset$ for any $1\leq i\leq k$.
In particular, we observe that $h_2^{\sigma}(\mu)$ reduces to the \emph{classical Cheeger constant} when $\sigma\in [\sigma_+]$.

\begin{Rmk}
The (multi-way) signed Cheeger constants provide new insights into existing constants reflecting connectivity (e.g., Cheeger constants) or bipartiteness (e.g., dual Cheeger constants, bipartiteness ratio of Trevisan, non-bipartiteness parameter of Desai and Rao) of unsigned graphs in the language of switching within the framework of signed graphs, thus giving a unified viewpoint about connectivity and bipartiteness of the underlying graph via assigning signatures.
\end{Rmk}

We can also define a natural family of \emph{antithetical dual signed Cheeger constants} $\{\widetilde{h}^{\sigma}_k(\mu), k=1,2,\ldots, n\}$ by $\widetilde{h}^{\sigma}_k(\mu):=h^{-\sigma}_k(\mu)$.  Dually, we have that
\begin{equation}\label{fact3}
\Gamma \text{ has an antibalanced connected component} \iff  \widetilde{h}_1^{\sigma}(\mu)=0.
\end{equation}

\section{Signed Cheeger inequality}
In this section, we prove the following signed Cheeger inequality.

\begin{Thm}\label{thmCheegerEsti}
Given a signed graph $\Gamma=(G, \sigma)$, we have
\begin{equation}\label{CheegerEsti}
\frac{\lambda_1(\Delta^{\sigma})}{2}\leq h^{\sigma}_1(\mu_d)\leq \sqrt{2\lambda_1(\Delta^{\sigma})}.
\end{equation}
\end{Thm}

The lower bound estimate in (\ref{CheegerEsti}) is easier.

\begin{proof}[Proof of the lower bound estimate of Theorem \ref{thmCheegerEsti}]
For any $(V_1, V_2)$, consider the particular function given by,
\begin{equation*}
f(u)=\left\{
       \begin{array}{rl}
         1, & \hbox{if $u\in V_1$;} \\
         -1, & \hbox{if $u\in V_2$;} \\
         0, & \hbox{otherwise.}
       \end{array}
     \right.
\end{equation*}
We calculate
\begin{align*}
\mathcal{R}^{\sigma}(f)
& =\frac{4|E^+(V_1, V_2)|+2|E^-(V_1)|+2|E^-(V_2)|+|E(V_1\cup V_2, \overline{V_1\cup V_2})|}{\text{vol}_{\mu}(V_1\cup V_2)}\\
& \leq 2\beta^{\sigma}(V_1, V_2).
\end{align*}
Then (\ref{varForleatLar}) implies $\lambda_1(\Delta^\sigma)\leq 2h_1^{\sigma}(\mu)$.
\end{proof}

The upper bound estimate in (\ref{CheegerEsti}) is more difficult.
The proof is based on the crucial observation that the estimate of $\lambda_1(\Delta^{\sigma})$ should be considered as a ``mixture" of the estimates of the smallest and largest eigenvalues of unsigned graphs (for which the smallest eigenvalue trivially equals $0$). This can be seen more clearly from the corresponding Rayleigh quotients.
One can appeal either to the techniques for proving the Cheeger inequality for unsigned graphs \cite{AM1985,Alon1986,Dodziuk1984} or to those for proving the dual Cheeger inequality \cite{Trevisan2012,BJ}.
For the former strategy, one needs first to switch the signature to be the one achieving the first minimum in (\ref{eq:signed expansion}).
We adopt here the latter strategy, for which we do not need to switch the signature to a proper one first. In fact, we will use the idea of Trevisan \cite{Trevisan2012} for proving the dual Cheeger inequality for unsigned graphs.

Given a non-zero function $f: V\rightarrow \mathbb{R}$ and a real number $t\geq 0$, define the following subsets of $V$:
\begin{equation*}
V_f(t):=\{u\in V: f(u)\geq t\}, \qquad V_f(-t):=\{u\in V: f(u)\leq -t\}.
\end{equation*}
Note that for $t=0$, we have the special definition $V_f(0):=\{u\in V: f(u)\geq 0\}$ and $V_f(-0):=\{u\in V: f(u)<0\}$.
Suppose $\max_{u\in V}f(u)^2=1$. For any $t\in [0,1]$, define a vector $$Y_f(t)\in \{-1,0,1\}^{V}$$ as follows: For each $u\in V$,
\begin{equation*}
(Y_f(t))_u:=\left\{
         \begin{array}{rl}
           1, & \hbox{if $u\in V_f(\sqrt{t})$;} \\
           -1, & \hbox{if $u\in V_f(-\sqrt{t})$;} \\
           0, & \hbox{otherwise.}
         \end{array}
       \right.
\end{equation*}
The following lemma is crucial for our purposes.

\begin{Lem}\label{lemmaDisc1}
For any $\{u,v\}\in E$,
\begin{equation}\label{lemmaForCheeger}
\int_0^1|(Y_f(t))_u-\sigma(uv)(Y_f(t))_v|\,dt \leq |f(u)-\sigma(uv)f(v)|(|f(u)|+|f(v)|).
\end{equation}
\end{Lem}

\begin{proof} It is enough to prove the estimate
\begin{equation}\label{lemmaForCheegerProof}
\int_0^1|(Y_f(t))_u-(Y_f(t))_v|\,dt \leq |f(u)-f(v)|(|f(u)|+|f(v)|)
\end{equation}
 for any non-zero functions $f: V \to [-1,1]$.
Then the estimate (\ref{lemmaForCheeger}) follows from applying (\ref{lemmaForCheegerProof}) to a function $g: V\to [-1,1]$ satisfying $g(u)=f(u)$ and $g(v)=\sigma(uv)f(v)$.

Now we prove (\ref{lemmaForCheegerProof}).
Without loss of generality, suppose $|f(u)|\geq |f(v)|$.
\begin{description}
  \item[Case 1] $f(u)$ and $f(v)$ have different signs. Then,
\begin{equation*}
|(Y_f(t))_u-(Y_f(t))_v|=\left\{
                  \begin{array}{ll}
                    2, & \hbox{if $t\leq f(v)^2$;} \\
                    1, & \hbox{if $f(v)^2< t\leq f(u)^2$;} \\
                    0, & \hbox{if $t> f(u)^2$.}
                  \end{array}
                \right.
\end{equation*}
Therefore,
\begin{align*}
 \int_0^1|(Y_f(t))_u-(Y_f(t))_v| \, dt &= f(u)^2+f(v)^2\leq (|f(u)|+|f(v)|)^2\\
 & = |f(u)-f(v)|(|f(u)|+|f(v)|).
\end{align*}

\item[Case 2] $f(u)$ and $f(v)$ have the same sign. In this case,
\begin{equation*}
|(Y_f(t))_u-(Y_f(t))_v|=\left\{
                  \begin{array}{ll}
                    0, & \hbox{if $t\leq f(v)^2$;} \\
                    1, & \hbox{if $f(v)^2< t\leq f(u)^2$;} \\
                    0, & \hbox{if $t> f(u)^2$.}
                  \end{array}
                \right.
\end{equation*}
Therefore,
\begin{align*}
 \int_0^1|(Y_f(t))_u-(Y_f(t))_v| \, dt&=f(u)^2-f(v)^2\\&=|f(u)-f(v)|\,(|f(u)|+|f(v)|).
\end{align*}
\end{description}
This completes the proof of (\ref{lemmaForCheegerProof}).
\end{proof}
\begin{Rmk}
In the proof of Lemma \ref{lemmaDisc1}, we used local level dualities to bring the two extremal cases together. We will use this principle again in the proofs of Lemma \ref{lemmaLocalization}, Lemma \ref{lemmaImprCheeger}, Lemma \ref{satzF}(ii), and Claim \ref{claim3} in later sections.
\end{Rmk}

In the following, we denote $$d_{\mu}^{w}:=\max_{u}\left\{\frac{\sum_{v,v\sim u}w_{uv}}{\mu(u)}\right\}.$$
\begin{Lem}\label{lemma:basicbrick}
For any non-zero function $f: V\rightarrow \mathbb{R}$, there exists a $t'\in [0, \max_{u\in V}f^2(u)]$ such that
\begin{equation}
 \beta^{\sigma}(V_f(\sqrt{t'}), V_f(-\sqrt{t'}))\leq \sqrt{2d_{\mu}^w\mathcal{R}^{\sigma}(f)}.
\end{equation}
\end{Lem}
\begin{proof}
Without loss of generality, we can assume $\max_{u\in V}f^2(u)=1$. Consider the ratio
$$I_f:=\frac{\int_0^1\sum_{u\sim v}w_{uv}|(Y_f(t))_u-\sigma(uv)(Y_f(t))_v|dt}{\int_0^1\sum_{u\in V}\mu(u)|(Y_f(t))_u|dt}.$$
By the definition of the vector $Y_f(t)$, we check that
\begin{align*}
 \sum_{u\sim v}w_{uv}&|(Y_f(t))_u-\sigma(uv)(Y_f(t))_v|\\
 =\; &2|E^+(V_f(\sqrt{t}), V_f(-\sqrt{t}))|+|E^-(V_f(\sqrt{t}))|+|E^-(V_f(-\sqrt{t}))|\\
 &+|E(V_f(\sqrt{t})\cup V_f(-\sqrt{t}), \overline{V_f(\sqrt{t})\cup V_f(-\sqrt{t})})|,
\end{align*}
and $$\sum_u\mu(u)|(Y_f(t))_u|=\text{vol}_{\mu}(V_f(\sqrt{t})\cup V_f(-\sqrt{t})).$$
Therefore, there exists $t'\in [0,1]$ such that
\begin{equation}\label{eq:ingredient1}
\beta^\sigma(V_f(\sqrt{t'}), V_t(-\sqrt{t'}))\leq I_f.
\end{equation}

On the other hand, we estimate
\begin{align*}
 I_f&\leq \dfrac{\sum_{u\sim v}w_{uv}|f(u)-\sigma(uv)f(v)|(|f(u)|+|f(v)|)}{\sum_u\mu(u)f(u)^2}\\[1ex]
&\leq \dfrac{\sqrt{\sum_{u\sim v}w_{uv}|f(u)-\sigma(uv)f(v)|^2} \;\sqrt{\sum_{u\sim v}w_{uv}(|f(u)|+|f(v)|)^2}}{\sum_u\mu(u)f(u)^2}.
\end{align*}
In the first inequality above, we used Lemma \ref{lemmaDisc1} and the fact that
$$\int_0^1|(Y_f(t))_u|\, dt=f(u)^2.$$
Observing that
\begin{align}\label{NorNonnor}\sum_{u\sim v}w_{uv}(|f(u)|+|f(v)|)^2&\leq \frac{1}{2}\sum_{u}\sum_{v,v\sim u}w_{uv}(2|f(u)|^2+2|f(v)|^2)\notag\\
&\leq 2d_\mu^w\sum_u\mu(u)f(u)^2,\end{align}
we arrive at
\begin{equation}\label{eq:ingredient2}
I_f\leq \sqrt{2d_\mu^w\mathcal{R}^{\sigma}(f)}.
\end{equation}
Combining (\ref{eq:ingredient1}) and (\ref{eq:ingredient2}) proves the lemma.
\end{proof}

\begin{proof}[Proof of the upper bound estimate of Theorem \ref{thmCheegerEsti}]
 Applying Lemma~\ref{lemma:basicbrick} to the first eigenfunction $\phi_1$ and choosing $\mu=\mu_d$, the upper bound estimate of (\ref{CheegerEsti}) is proved.
\end{proof}

For the signed non-normalized Laplace matrix $L^{\sigma}$, we have the following Cheeger-type estimate.
\begin{Thm}\label{thmCheegerNonnormalized}
Let $\Gamma=(G, \sigma)$ be a signed graph. Then
\begin{equation}\label{CheegerEstiNonnormalized}
\frac{\lambda_1(L^{\sigma})}{2}\leq  h_1^{\sigma}(\mu_1)\leq \sqrt{2d_{\max}\lambda_1(L^{\sigma})}.
\end{equation}
where $d_{\max}:=\max_{u\in V}d_u$.
\end{Thm}
\begin{proof}
Applying Lemma \ref{lemma:basicbrick} to the first eigenfunction of $L^\sigma$ and choosing $\mu=\mu_1$ yield the upper bound estimate. The lower bound estimate follows similarly as in Theorem \ref{thmCheegerEsti}.
\end{proof}

%

\section{Higher-order signed Cheeger inequalities}
In this section, we prove higher-order versions of the signed Cheeger inequality \eqref{CheegerEsti}.

\begin{Thm}\label{HigerCheeger}
There exists an absolute constant $C$ such that for any signed graph $\Gamma=(G,\sigma)$ and any $k \in \{1,2,\dots,N\}$,
\begin{equation}\label{eq:higherCHeegerIn}
\frac{\lambda_k(\Delta^{\sigma})}{2}\leq h_k^{\sigma}(\mu_d)\leq Ck^3\sqrt{\lambda_k(\Delta^{\sigma})}.
\end{equation}
\end{Thm}

This is a generalization of the higher-order Cheeger and dual Cheeger inequalities for unsigned graphs by Lee, Oveis Gharan, and Trevisan \cite{LOT2013} and the second named author \cite{Liu13}.

\subsection{The lower bound estimate} Again, the lower bound estimate in (\ref{eq:higherCHeegerIn}) is easier to show.

\begin{proof}[Proof of the lower bound estimate of Theorem \ref{HigerCheeger}]

For any  $$\{(V_{2i-1}, V_{2i})\}_{i=1}^k\in \mathrm{Pair}(k),$$ that is, any $k$-sub-bipartition
of $V$,  define for each $i$ a function
\begin{equation*}
f_i(u):=\left\{
         \begin{array}{rl}
           1, & \hbox{if $u\in V_{2i-1}$;} \\
           -1, & \hbox{if $u\in V_{2i}$;} \\
           0, & \hbox{otherwise.}
         \end{array}
       \right.
\end{equation*}
Let $f=\sum_{i=1}^ka_if_i$, where $a_1, \ldots, a_k\in \mathbb{R}$, be a function in the space $\mathrm{span}\{f_1, \ldots, f_k\}$.
Consider the signed Rayleigh quotient
\begin{equation*}
 \mathcal{R}^{\sigma}(f)=\frac{\sum_{u\sim v}w_{uv}(f(u)-\sigma(uv)f(v))^2}{\sum_{u}\mu(u)f(u)^2}
\end{equation*}
of $f$.
For the denominator, we have
\begin{equation*}
 \sum_u\mu(u)f(u)^2=\sum_{i=1}^ka_i^2\sum_u\mu(u)f_i(u)^2=\sum_{i=1}^ka_i^2\mathrm{vol}_{\mu}(V_{2i-1}\cup V_{2i}).
\end{equation*}
And for the numerator, we estimate
\begin{align*}
\sum_{u\sim v}w_{uv}&(f(u)-\sigma(uv)f(v))^2 \\
\leq  2\sum_{i=1}^ka_i^2\big(&2|E^{+}(V_{2i-1}, V_{2i})|+|E^-(V_{2i-1})|
+|E^-(V_{2i-1})|\\&+|E(V_{2i-1}\cup V_{2i}, \overline{V_{2i-1}\cup V_{2i}})|\big).
\end{align*}
Hence, we arrive at
\begin{equation}
 \max_{a_1,\ldots, a_k}\mathcal{R}^{\sigma}(f)\leq 2\max_{1\leq i\leq k}\beta^{\sigma}(V_{2i-1}, V_{2i}).
\end{equation}
By the min-max principle (\ref{keigenvalue}), this implies $\lambda_k\leq 2h_k^{\sigma}(\mu)$.
\end{proof}

\subsection{The upper bound estimate}
We next prove the remaining upper bound estimate of (\ref{eq:higherCHeegerIn}).

\subsubsection{Idea of the proof: a signed spectral clustering algorithm}\label{section:clustering}
Theorem \ref{HigerCheeger} is a mixture of the higher-order Cheeger \cite{LOT2013} and dual Cheeger \cite{Liu13} inequalities, the proofs of which utilize spectral clustering algorithms via metrics on spheres and real projective spaces, respectively. One might anticipate at first that the proper metrics for proving the upper bound estimate in Theorem \ref{HigerCheeger} are a mixture of those two kinds of metrics. It is somewhat surprising that the latter metrics \cite{Liu13} themselves are competent for the proof and provide unified spectral clustering algorithms.


We describe the \emph{signed spectral clustering algorithm} in detail. This algorithm aims at finding $k$ subsets whose
induced subgraphs are nearly balanced. The connections among those $k$ subsets, regardless of their signs, are very sparse.

Let $\{\phi_1, \phi_2, \ldots, \phi_n\}$ be an orthonormal system of eigenfunctions corresponding to $\lambda_1(\Delta^{\sigma}), \lambda_2(\Delta^{\sigma}), \ldots, \lambda_n(\Delta^{\sigma})$.

\begin{itemize}
  \item [(1)] \emph{Spectral embedding.} Using the first $k$ eigenfunctions, we obtain a coordinate system for the vertices via the map
\begin{align*}
 \Phi: V\rightarrow \mathbb{R}^k, \,\,\,v\mapsto (\phi_1(v), \phi_2(v), \ldots, \phi_k(v)).
\end{align*}
\item[(2)] \emph{Normalization.} We further map $\widetilde{V}_{\Phi}:=\{v: \Phi(v)\neq 0\}$ to the unit sphere,
\begin{align*}
 \Phi^{\mathrm{nor}}: \widetilde{V}_{\Phi}\rightarrow \mathbb{S}^{k-1}, \,\,\,v\mapsto \frac{\Phi(v)}{\Vert\Phi(v)\Vert}.
\end{align*}
  \item [(3)] \emph{Clustering the points.} We use the following pseudometric $d_{\Phi}$ on $\widetilde{V}_{\Phi}$ studied in \cite{Liu13},
\begin{equation}\label{eq:projmetric}
 d_{\Phi}(u,v):=\min\left\{\left\Vert\Phi^{\mathrm{nor}}(u)+\Phi^{\mathrm{nor}}(v)\right\Vert, \left\Vert\Phi^{\mathrm{nor}}(u)-\Phi^{\mathrm{nor}}(v)\right\Vert\right\},
\end{equation}
where $\Vert\cdot\Vert$ stands for the Euclidean norm in $\mathbb{R}^k$.
\end{itemize}
Recall that the projective space $P^{k-1}\mathbb{R}$ is obtained from $\mathbb{S}^{k-1}$ by identifying the antipodal points,
\begin{equation*}
 Pr: \mathbb{S}^{k-1} \rightarrow P^{k-1}\mathbb{R}: x, -x\mapsto [x],
\end{equation*} where $x$ are the unit vectors in $\mathbb{R}^k$.
The metric (\ref{eq:projmetric}) is induced from the following metric on $P^{k-1}\mathbb{R}$,
\begin{equation*}
 d([x],[y]):=\min\{\Vert x+y\Vert, \Vert x-y\Vert\}, \,\, \forall\,\, [x], [y]\in P^{k-1}\mathbb{R}.
\end{equation*}

If $\sigma=\sigma_+$, we have $\lambda_1(\Delta^{\sigma})=0$ and $\phi_1$ is the constant function $\phi_1\equiv 1/\sqrt{\mathrm{vol}_{\mu_d}(V)}$.
Our algorithm reduces to the traditional clustering for finding $k$ subsets, each defining a sparse cut \cite{Luxburg07,NJW2002}. A difference is that, traditionally, the spherical metric (or the radial projection distance) $d^{\mathrm{sphere}}_{\Phi}(u,v):= \left\Vert\Phi^{\mathrm{nor}}(u)-\Phi^{\mathrm{nor}}(v)\right\Vert$ is used for clustering the points, as verified by Lee, Oveis Gharan, and Trevisan \cite{LOT2013}. Here in our case, we use the metric (\ref{eq:projmetric}) instead.


If, on the other hand,  $\sigma=\sigma_-$, our algorithm  reduces to finding $k$ almost-bipartite subgraphs, since $(G, \sigma_-)$ is balanced if and only if $G$ is bipartite. This is exactly the one proposed in \cite{Liu13}.

Theorem \ref{HigerCheeger} and the proof presented below provide the worst-case performance guarantee of the algorithm described above.

\begin{Rmk}
If we use the last $k$ eigenfunctions $$\phi_{N-k+1}, \phi_{N-k+2}, \ldots, \phi_N$$ instead of the first $k$ eigenfunctions in the step of spectral embedding, we will obtain an algorithm for finding $k$ subsets whose induced subgraphs are nearly antibalanced, each defining a sparse cut.
\end{Rmk}

\begin{Rmk}
We comment here about further related works on signed spectral algorithms. For any signed graph (or subgraph), one can continue to do the next-level clustering. Roughly speaking, the objective is to find two subsets whose signed bipartiteness ratio is small. The heuristics of the spectral method for such clustering was discussed in \cite{KSLLDA,Kunegis14,CGVZ13}.
Actually, the proof of Theorem \ref{thmCheegerEsti} (especially Lemma \ref{lemma:basicbrick}) provides a theoretical guarantee for their heuristic arguments. We can achieve this clustering by the threshold sets $V_f(t):=\{u\in V: f(u)\geq t\}$ and $V_f(-t):=\{u\in V: f(u)\leq -t\}$ of a certain function $f$.

There are studies about another kind of multi-way clustering of signed networks, called the correlation clustering. It aims at finding $k$ non-trivial disjoint subsets $V_1, V_2, \ldots, V_k$ such that edges connecting two vertices from the same subset are almost all positive and edges connecting two vertices from different subsets are almost all negative. Heuristic spectral algorithms for such clustering were studied in, e.g., \cite{KLB09,KSLLDA,Kunegis14,SM13,SIMM13,CGVZ13}; for non-spectral algorithms, see, e.g., \cite{Swamy04}.
\end{Rmk}

\subsubsection{The proof}
Let $\phi_1, \phi_2, \ldots, \phi_n$ and $\Phi$ be defined as above.
Since $\lambda_i=\mathcal{R}^{\sigma}(\phi_i)$, $i=1,2,\ldots, k$, we have
\begin{equation}\label{eq:lambdakReyleigh}
 \lambda_k\geq \frac{\sum_{u\sim v}w_{uv}\Vert \Phi(u)-\sigma(uv)\Phi(v)\Vert^2}{\sum_{u\in V}\mu(u)\Vert\Phi(u)\Vert^2}=\mathcal{R}^{\sigma}(\Phi).
\end{equation}
In the following, we will find $k$ disjointly supported maps $\{\Psi_i\}_{i=1}^k$ by localizing $\Phi$, such that $\mathcal{R}^{\sigma}(\Psi_i)$ can be bounded from above by $\mathcal{R}^{\sigma}(\Phi_i)$ (up to a polynomial in $k$).

Recall the pseudometric space $(\widetilde{V}_{\Phi}, d_{\Phi})$ from the algorithm described in the last subsection. In order to localize $\Phi$, we consider the following cut-off function: For any $S_i\subseteq V$ and $\epsilon>0$, define
\begin{equation*}
 \theta_i(v)=\left\{
               \begin{array}{ll}
                 0, & \hbox{if $\Phi(v)=0$;} \\
                 \max\left\{0, \, 1-\dfrac{d_{\Phi}(v,S_i\cap \widetilde{V}_{\Phi})}{\epsilon}\right\}, & \hbox{otherwise.}
               \end{array}
             \right.
\end{equation*}
Then we can localize $\Phi$ as
\begin{equation*}
 \Psi_i:=\theta_i \Phi: V\rightarrow \mathbb{R}^k.
\end{equation*}
Observe that $\Psi_i|_{S_i}=\Phi|_{S_i}$ and $$\mathrm{supp}(\Psi_i)\subseteq N_{\epsilon}(S_i\cap \widetilde{V}_{\Phi}, d_{\Phi}):=\{v\in \widetilde{V}_{\Phi}: d_{\Phi}(v, S_i\cap \widetilde{V}_{\Phi})<\epsilon\}.$$
The following crucial lemma
is an extension of  \cite[Lemma 5.3]{Liu13} and \cite[Lemma 3.3]{LOT2013}.

\begin{Lem}\label{lemmaLocalization}
For any given $0<\epsilon <2$, define $\Psi_i$ as above. Then for any $\{u,v\}\in E$,
\begin{equation}\label{eq:Localization}
 \Vert\Psi_i(u)-\sigma(uv)\Psi_i(v)\Vert\leq \left(1+\frac{2}{\epsilon}\right)\Vert\Phi(u)-\sigma(uv)\Phi(v)\Vert.
\end{equation}
\end{Lem}

\begin{proof}
If either $\Phi(u)$ or $\Phi(v)$ vanishes, (\ref{eq:Localization}) follows from the fact that $|\theta_i|\leq 1$. So we only need to prove (\ref{eq:Localization}) for $u,v\in \widetilde{V}_{\Phi}$. In this case, we calculate
\begin{align}
\Vert\Psi_i(u)&-\sigma(uv)\Psi_i(v)\Vert =\Vert\theta_i(u)\Phi(u)-\sigma(uv)\theta_i(v)\Phi(v)\Vert\notag\\
& \leq |\theta_i(u)| \, \Vert\Phi(u)-\sigma(uv)\Phi(v)\Vert + |\theta_i(u)-\theta_i(v)| \, \Vert\Phi(v)\Vert.\label{eq:preclain}
\end{align}
We claim that
\begin{equation}\label{eq:claim}
 |\theta_i(u)-\theta_i(v)| \, \Vert\Phi(v)\Vert\leq\frac{2}{\epsilon}\Vert\Phi(u)-\sigma(uv)\Phi(v)\Vert.
\end{equation}
Note that (\ref{eq:Localization}) follows immediately from (\ref{eq:preclain}) and (\ref{eq:claim}).
Hence, the remaining task is to prove (\ref{eq:claim}). Similarly to the beginning of the proof of Lemma \ref{lemmaDisc1}, it is enough to show
\begin{equation}\label{eq:afterclaim}
 |\theta_i(u)-\theta_i(v)| \, \Vert\Phi(v)\Vert\leq\frac{2}{\epsilon}\Vert\Phi(u)-\Phi(v)\Vert
\end{equation}
for any two vectors $\Phi(u), \Phi(v)\in \mathbb{R}^k\setminus\{0\}$.
\begin{description}
  \item[Case 1] The edge $\{u,v\}$ satisfies $$d_{\Phi}(u,v)=\left\Vert\frac{\Phi(u)}{\Vert \Phi(u)\Vert}-\frac{\Phi(v)}{\Vert \Phi(v)\Vert}\right\Vert.$$
       In this case $\langle \Phi(u), \Phi(v)\rangle\geq 0$, where $\langle\cdot, \cdot\rangle$ stands for the inner product of  $\mathbb{R}^k$. We estimate
\begin{align*}
 |\theta_i(u)-\theta_i(v)|\Vert\Phi(v)\Vert&\leq\frac{1}{\epsilon}d_{\Phi}(u,v)\Vert\Phi(v)\Vert=\frac{1}{\epsilon}\left\Vert\frac{\Vert \Phi(v)\Vert}{\Vert \Phi(u)\Vert}\Phi(u)-\Phi(v)\right\Vert\\
& \leq  \frac{1}{\epsilon}\Vert \Phi(u)-\Phi(v)\Vert+\frac{1}{\epsilon}\left\Vert\frac{\Vert \Phi(v)\Vert}{\Vert \Phi(u)\Vert}\Phi(u)-\Phi(u)\right\Vert\\
& \leq \frac{1}{\epsilon}\Vert \Phi(u)-\Phi(v)\Vert+\frac{1}{\epsilon}\left|\Vert \Phi(v)\Vert-\Vert\Phi(u)\Vert\right|\\
&\leq\frac{2}{\epsilon}\Vert\Phi(u)-\Phi(v)\Vert.
\end{align*}
  \item[Case 2] The edge $\{u,v\}\in E$ satisfies $$d_{\Phi}(u,v)=\left\Vert\frac{\Phi(u)}{\Vert \Phi(u)\Vert}+\frac{\Phi(v)}{\Vert \Phi(v)\Vert}\right\Vert.$$
In this case $\langle \Phi(u), \Phi(v)\rangle\leq 0$. Thus,
\begin{equation*}
|\theta_i(u)-\theta_i(v)| \, \Vert\Phi(v)\Vert\leq\Vert \Phi(v)\Vert\leq \Vert \Phi(u)-\Phi(v)\Vert.
\end{equation*}
\end{description}
This completes the proof of (\ref{eq:afterclaim}).
\end{proof}

 In fact, we can find $k$ disjoint subsets of $\widetilde{V}_{\Phi}$ with good properties.
\begin{Lem}\label{lemma:goodpartitions}
There exist $k$ non-empty, mutually disjoint subsets $$S_1, S_2, \ldots, S_k\subseteq \widetilde{V}_{\Phi}$$ and an absolute constant $C_0>1$, such that
\begin{itemize}
  \item for any $1\leq i\neq j\leq k$, \begin{equation}
  d_{\Phi}(S_i, S_j)\geq \frac{1}{C_0k^{\frac{5}{2}}};
\end{equation}
  \item for any $1\leq i\leq k$,
\begin{equation}
 \sum_{u\in S_i}\mu(u)\Vert\Phi(u)\Vert^2\geq\frac{1}{2k} \sum_{u\in V}\mu(u)\Vert\Phi(u)\Vert^2.
\end{equation}
\end{itemize}
\end{Lem}
The proof of Lemma \ref{lemma:goodpartitions} employs the padded random partition theory on $(\widetilde{V}_{\Phi}, d_{\Phi})$.
Since the signature plays no role in this lemma, we refer to \cite[Section 6]{Liu13} for the proof.
\begin{Rmk}\label{rmk:C0}
In fact, from the arguments in \cite[Proof of Theorem 6.1]{Liu13}, we can set $C_0=768\left(\log_2\pi-\frac{1}{2}\right)$.
\end{Rmk}

Combining Lemma \ref{lemmaLocalization} and Lemma \ref{lemma:goodpartitions} leads to the following result.

\begin{Lem}\label{lemma:localReyleighquotient}
For any $k \in \{1,2,\dots, n\}$, there exist $k$ disjointly supported functions $\psi_1, \psi_2, \ldots, \psi_k: V\rightarrow \mathbb{R}$ such that for each $1\leq i\leq k$,
\begin{equation}\label{eq:5.9}
 \mathcal{R}^{\sigma}(\psi_i)\leq Ck^6\mathcal{R}^{\sigma}(\Phi),
\end{equation}
where $C$ is an absolute constant.
\end{Lem}

\begin{proof} Let $\{\theta_i\}_{i=1}^k$ be the $k$ cut-off functions corresponding to $\{S_i\}_{i=1}^k$ obtained in Lemma \ref{lemma:goodpartitions},
 and set $\epsilon=1/(2C_0k^{\frac{5}{2}})$. For each $i$, define $\Psi_i=\theta_i  F$. By Lemma \ref{lemmaLocalization},
\begin{equation*}
 \sum_{u\sim v}w_{uv}\Vert\Psi_i(u)-\sigma(uv)\Psi_i(v)\Vert^2\leq \left(1+\frac{2}{\epsilon}\right)^2\sum_{u\sim v}w_{uv}\Vert\Phi(u)-\sigma(uv)\Phi(v)\Vert^2.
\end{equation*}
By Lemma \ref{lemma:goodpartitions},
\begin{equation*}
 \sum_{u\in V}\mu(u)\Vert\Psi_i(u)\Vert^2\geq\frac{1}{2k}\sum_{u\in V}\mu(u)\Vert\Phi(u)\Vert^2.
\end{equation*}
Therefore,
\begin{equation}\label{eq:5.12}
 \mathcal{R}^{\sigma}(\Psi_i)\leq 2k(1+2C_0k^{\frac{5}{2}})^2\mathcal{R}^{\sigma}(\Phi)\leq Ck^6\mathcal{R}^{\sigma}(\Phi).
\end{equation}
Let us write $\Psi_i=(\Psi_i^1, \Psi_i^2, \ldots, \Psi_i^k): V\rightarrow \mathbb{R}^k$. Notice that we can always find $j_0\in \{1,2,\ldots,k\}$ such that
\begin{equation*}
 \mathcal{R}^{\sigma}(\Psi_i^{j_0})\leq \mathcal{R}^{\sigma}(\Psi_i).
\end{equation*}
Setting $\psi_i:=\Psi_i^{j_0}$, $1\leq i\leq k$ yields (\ref{eq:5.9}).
\end{proof}
\begin{proof}[Proof of the upper bound estimate in Theorem \ref{HigerCheeger}]
Let us assign $\mu=\mu_d$. Then combining Lemma \ref{lemma:basicbrick}, (\ref{eq:lambdakReyleigh}) and Lemma \ref{lemma:localReyleighquotient} together leads to the estimate
 $$h_k^\sigma(\mu_d)\leq \sqrt{2C} k^3\sqrt{\lambda_k(\Delta^\sigma)},$$
 where $C$ is the constant in (\ref{eq:5.9}).
This proves the upper bound estimate in (\ref{eq:higherCHeegerIn}).
\end{proof}
\begin{Rmk}
By (\ref{eq:5.12}), we observe that we can set the constant $C$ in in (\ref{eq:5.9}) to be $C=2(2C_0+1)^2$. Then the proof above actually leads to
$$h_k^\sigma(\mu_d)\leq 2(2C_0+1) k^3\sqrt{\lambda_k(\Delta^\sigma)}.$$
Inserting the constant from Remark \ref{rmk:C0}, we have
\begin{equation}
h_k^\sigma(\mu_d)\leq 3540 k^3\sqrt{\lambda_k(\Delta^\sigma)}.
\end{equation}
In this way, we can also get explicit constants in Theorem \ref{HigerCheegerNonnormal}, Theorem \ref{ImprChIntroHiger}, Corollary \ref{CorImprChNonHiger}, in terms of $C_0$ from Remark \ref{rmk:C0}. As those constant are apparently not optimal, we will not pursue that in this paper.
\end{Rmk}

If we instead assign $\mu=\mu_1$ in the above proof, we obtain the following estimate for $L^{\sigma}$:

\begin{Thm}\label{HigerCheegerNonnormal}
There exists an absolute constant $C$ such that for any signed graph $\Gamma=(G,\sigma)$ and any $k \in \{1,2,\dots, N\}$,
\begin{equation}\label{eq:higherCHeegerInnonnormal}
\frac{\lambda_k(L^{\sigma})}{2}\leq h_k^{\sigma}(\mu_1)\leq Ck^3\sqrt{d_{\max}\lambda_k(L^{\sigma})}.
\end{equation}
\end{Thm}

\section{Signed Cheeger constants and higher order spectral gaps}

A natural question is, whether the order of $\lambda_1(\Delta^{\sigma})$ on the right hand side of the signed Cheeger inequality (\ref{CheegerEsti}) can be improved to be $1$. Extending the ideas of Kwok et al.~\cite{KLLOT2013}, we answer this question by the following theorem.

\begin{Thm}\label{ImprChIntro}
For any signed graph $\Gamma=(G,\sigma)$ and any $k \in \{1,2,\dots,n\}$,
\begin{equation} \label{qq22}
 h_1^{\sigma}(\mu_d)<16\sqrt{2}k\frac{\lambda_1(\Delta^{\sigma})}{\sqrt{\lambda_k(\Delta^{\sigma})}}.
\end{equation}
\end{Thm}

In other words, when there exists a $k$ such that the gap between $\lambda_1$ and $\lambda_k$ is large, one can improve the order of $\lambda_1$ on the right hand side of (\ref{CheegerEsti}) to be $1$. Actually, a slightly stronger version of \eqref{qq22} can be proved:



\begin{Thm} \label{ImprCh}
Given a signed graph $\Gamma=(G,\sigma)$ and $k \in \{1,2,\dots, N\}$, at least one of the following holds,
\begin{equation}
\text{(i). } h_1^{\sigma}(\mu_d)\leq 8k\lambda_1(\Delta^{\sigma}); \quad
\text{(ii). } h_1^{\sigma}(\mu_d)< 16\sqrt{2}k\frac{\lambda_1(\Delta^{\sigma})}{\sqrt{\lambda_k(\Delta^{\sigma})}}.
\end{equation}
\end{Thm}
Notice that Theorem \ref{ImprChIntro} is a direct corollary of this theorem
and the fact that $0\leq \lambda_k(\Delta^{\sigma})\leq 2$.

Along the way to proving Theorem~\ref{ImprCh}, the following lemma will be crucial, which should be compared with Lemma \ref{lemma:basicbrick}.

\begin{Lem}\label{lemmaImprCheeger}
 For any non-zero function $f:V\rightarrow \mathbb{R}$, there exists a $t'\in [0, \max_{u\in V}|f(u)|]$ such that
\begin{equation}\label{preImprCheeger}
 \beta^{\sigma}(V_f(t'),V_f(-t'))\leq \frac{\sum_{u\sim v}w_{uv}|f(u)-\sigma(uv)f(v)|}{\sum_u\mu(u)|f(u)|}.
\end{equation}
\end{Lem}

\begin{proof}
We can assume $\max_{u\in V}|f(u)|=1$ without loss of generality since the right hand side of (\ref{preImprCheeger}) in invariant under scaling of $f$. For $t\in [0,1]$, define a vector $X_f(t)\in \{-1,0,1\}^V$ by
\begin{equation}
 (X_f(t))_u=\left\{
       \begin{array}{rl}
         1, & \hbox{if $f(u)\geq t$;} \\
         -1, & \hbox{if $f(u)\leq -t$;} \\
         0, & \hbox{otherwise.}
       \end{array}
     \right.
\end{equation}
We claim that, for any $\{u,v\}\in E$,
\begin{equation}\label{claim1}
 \int_0^1|(X_f(t))_u-\sigma(uv)(X_f(t))_v|dt=|f(u)-\sigma(uv)f(v)|.
\end{equation}
Similarly to the proof of Lemma \ref{lemmaDisc1}, we only need to show that
\begin{equation}\label{eq:6.5}
  \int_0^1|(X_f(t))_u-(X_f(t))_v|dt=|f(u)-f(v)|
\end{equation}
for any function $f: V \to [-1,1]$.
Without loss of generality, suppose $|f(u)|\geq |f(v)|$. If $f(u)$ and $f(v)$ have different signs, then
\begin{equation}
 |(X_f(t))_u-(X_f(t))_v|=\left\{
                   \begin{array}{ll}
                     2, & \hbox{if $t\leq |f(v)|$;} \\
                     1, & \hbox{if $|f(v)|< t\leq|f(u)|$;} \\
                     0, & \hbox{if $t> |f(u)|$,}
                   \end{array}
                 \right.
\end{equation}
and so,
$$\int_0^1|(X_f(t))_u-(X_f(t))_v|\,dt=|f(u)|+|f(v)|=|f(u)-f(v)|.$$
If, on the other hand, $f(u)$ and $f(v)$ have the same sign,
\begin{equation}
|(X_f(t))_u-(X_f(t))_v|=\left\{
                   \begin{array}{ll}
                     0, & \hbox{if $t\leq |f(v)|$;} \\
                     1, & \hbox{if $|f(v)|< t\leq|f(u)|$;} \\
                     0, & \hbox{if $t> |f(u)|$,}
                   \end{array}
                 \right.
\end{equation}
and so,
$$\int_0^1|(X_f(t))_u-(X_f(t))_v|dt=|f(u)|-|f(v)|=|f(u)-f(v)|.$$
This completes the proof of (\ref{eq:6.5}) and establishes the claim (\ref{claim1}).

By the definition of the vector $X_f(t)$, we observe that
\begin{align*}
 \sum_{u\sim v}w_{uv}&|(X_f(t))_u-\sigma(uv)(X_f(t))_v|\\
 =\; &2|E^+(V_f(t), V_f(-t))|+|E^-(V_f(t))|+|E^-(V_f(-t))|\\&+|E(V_f(t)\cup V_f(-t), \overline{V_f(t)\cup V_f(-t)})|,
\end{align*}
and $$\sum_u\mu(u)|(X_f(t))_u|=\text{vol}_{\mu}(V_f(t)\cup V_f(-t)).$$
Therefore, there exists a $t'\in [0,1]$ such that
\begin{align*}
 \beta^{\sigma}(V_f(t'), V_f(-t'))\leq\frac{\int_0^1\sum_{u\sim v}w_{uv}|(X_f(t))_u-\sigma(uv)(X_f(t))_v|dt}{\int_0^1\sum_u\mu(u)|(X_f(t))_u|dt}.
\end{align*}
Applying the claim (\ref{claim1}) and the fact $\int_0^1|(X_f(t))_u|dt=|f(u)|$, the lemma is proved.
\end{proof}

For any non-zero function $f:V\rightarrow \mathbb{R}$, and for any $k \in \mathbb{N}$, let
\begin{equation}
 0=t_0\leq t_1\leq\cdots\leq t_{2k}
\end{equation}
be a sequence of real numbers with $t_{2k}=\max_{u\in V}|f(u)|$. We define a step function approximation $g$ to $f$ as
\begin{equation}\label{DefOFg}
 g(u)=\psi_{-t_{2k},\ldots, -t_1, 0, t_1, \ldots, t_{2k}}(f(u)):=\arg\min_{t\in\{-t_{2k},\ldots, 0, \ldots, t_{2k}\}}|f(u)-t|.
\end{equation}
In other words, $g: V \to \{-t_{2k},\ldots, 0, \ldots, t_{2k}\}$ is a function such that $g(u)$ equals the one of $\{-t_{2k},\ldots, 0, \ldots, t_{2k}\}$ which is closest to $f(u)$.

We further construct an auxiliary function $F: V\rightarrow \mathbb{R}$. First, we define a function $\eta: [-t_{2k}, t_{2k}]\rightarrow \mathbb{R}$ via
\begin{equation}
 \eta(x):=|x-\psi_{-t_{2k},\ldots, -t_1, 0, t_1, \ldots, t_{2k}}(x)|.
\end{equation}
Note that $\eta(-x)=\eta(x)$. Then for each $u\in V$, we assign
\begin{equation} \label{F}
 F(u):=\int_0^{f(u)}\eta(x) \, dx.
\end{equation}

\begin{Lem}\label{satzF}
The function $F$ defined in \eqref{F} has the following properties:
\begin{itemize}
  \item[(i)] For any $u\in V$,
\begin{equation}
 |F(u)|\geq \frac{1}{8k}|f(u)|^2.
\end{equation}
  \item[(ii)]For any $\{u,v\}\in E$,
\begin{align}
 &|F(u)-\sigma(uv)F(v)| \notag\\
 \leq & \frac{1}{2}|f(u)-\sigma(uv)f(v)|\times \notag\\
 &\hspace{.5cm}\Big(|f(u)-\sigma(uv)f(v)|+|f(u)-g(u)|+|f(v)-g(v)|\Big).\label{goalPro}
\end{align}
\end{itemize}
\end{Lem}

\begin{proof}
 (i). First observe that $F(u)$ and $f(u)$ share the same sign. Without loss of generality, assume $f(u)> 0$. Then the proof can be done as in \cite[Claim 3.3]{KLLOT2013}. For the readers' convenience, we recall it here. Suppose $f(u)\in [t_i, t_{i+1}]$ for some $i$. Then by the Cauchy-Schwarz inequality,
\begin{align*}
f(u)^2&=\left(\sum_{j=0}^{i=1}(t_{j+1}-t_j)+(f(u)-t_i)\right)^2\\
&\leq 2k\left(\sum_{j=0}^{i=1}(t_{j+1}-t_j)^2+(f(u)-t_i)^2\right).
\end{align*}
By the definition of $F$,
\begin{align*}
F(u)&=\sum_{j=0}^{i-1}\int_{t_j}^{t_{j+1}}\eta(x)\,dx+\int_{t_i}^{f(u)}\eta(x)\,dx\\
&\geq \sum_{j=0}^{i-1}\frac{1}{4}(t_{j+1}-t_j)^2+\frac{1}{4}(f(u)-t_i)^2=\frac{1}{8k}f(u)^2.
\end{align*}
(ii). Observe that $$\int_0^{\sigma(uv)f(v)}\eta(x) \, dx=\sigma(uv)\int_0^{f(v)}\eta(x) \, dx.$$  Hence, we only need to prove (\ref{goalPro}) for the case $\sigma(uv)=+1$.
Without loss of generality, suppose $|f(u)|\geq |f(v)|$. If $f(u)$ and $f(v)$ have different signs, say, $f(u)\geq 0$ and $f(v)\leq 0$, then, recalling the fact $\eta(-x)=\eta(x)$, we have
\begin{align}
 |F(u)-F(v)|&=\int_0^{f(u)}\eta(x)\,dx+\int_0^{-f(v)}\eta(x)\,dx\notag\\
& \leq \int_0^{f(u)}x\,dx+\int_0^{-f(v)}x\,dx\notag\\
&=\frac{1}{2}(f(u)^2+f(v)^2)\leq \frac{1}{2}|f(u)-f(v)|^2.\label{eq:1}
\end{align}
If, on the other hand, $f(u)$ and $f(v)$ have the same sign, we can assume $f(u)>0$ and $f(v)>0$ since $\eta(-x)=\eta(x)$ and the step function approximation of $-f$ is $-g$. Then,
\begin{align}\label{eq:2}
 |F(u)-F(v)|=\int_{f(v)}^{f(u)}\eta(x)\,dx\leq |f(u)-f(v)|\times \max_{f(v)\leq x\leq f(u)}\eta(x).
\end{align}
Using the definition of $\eta$, we estimate
\begin{align}
 \eta(x)&\leq\min\{|x-g(u)|,|x-g(v)|\}\leq \frac{1}{2}(|x-g(u)|+|x-g(v)|)\notag\\
&\leq \frac{1}{2}
\left(|x-f(u)|+|f(u)-g(u)|+|x-f(v)|+|f(v)-g(v)|\right)\notag\\
&=\frac{1}{2}
\left(|f(u)-f(v)|+|f(u)-g(u)|+|f(v)-g(v)|\right).\label{eq:3}
\end{align}
Combining (\ref{eq:1}), (\ref{eq:2}), and (\ref{eq:3}) leads to (\ref{goalPro}).
\end{proof}

With Lemma \ref{satzF} and Lemma \ref{lemmaImprCheeger} at hand, we derive the following lemma.
\begin{Lem}\label{preImprCheeger2}
 For any non-zero function $f: V\rightarrow \mathbb{R}$ and any step function approximation $g$ of $f$ constructed from $0=t_0\leq t_1, \ldots, t_{2k}$ as above, there exists a $t'\in [0, \max_{u\in V}|f(u)|]$ such that
\begin{equation}\label{CheegerToRayleighRefine}
 \beta^{\sigma}(V_f(t'), V_{-f}(t'))\leq 4k\mathcal{R}^{\sigma}(f)+4\sqrt{2}k\frac{\Vert f-g\Vert_{\mu}}{\Vert f\Vert_{\mu}}\sqrt{d_{\mu}^w\mathcal{R}^{\sigma}(f)},
\end{equation}where $\Vert f\Vert^2_{\mu}:=\sum_{u\in V}\mu(u)f(u)^2$.
\end{Lem}
We point out that the notation $\Vert\cdot\Vert_{\mu_1}=\Vert\cdot\Vert$ reduces to the Euclidean norm when $\mu=\mu_1$.
\begin{proof}
Applying Lemma \ref{lemmaImprCheeger} to the function $F$, we find $\overline{t}\in [0, \max_u|F(u)|]$ such that
\begin{align*}
 &\beta^{\sigma}(V_F(\overline{t}), V_F(-\overline{t}))\leq \frac{\sum_{u\sim v}w_{uv}|F(u)-\sigma(uv)F(v)|}{\sum_u\mu(u)|F(u)|}\\
\leq &4k\mathcal{R}^{\sigma}(f)\\
&+4k\frac{\sum_{u\sim v}w_{uv}|f(u)-\sigma(uv)f(v)|(|f(u)-g(u)|+|f(v)-g(v)|)}{\sum_u\mu(u)f(u)^2}\\[1ex]
\leq & 4k\mathcal{R}^{\sigma}(f)+4k\sqrt{\mathcal{R}^{\sigma}(f)}\frac{\sqrt{\sum_{u\sim v}w_{uv}|(|f(u)-g(u)|+|f(v)-g(v)|)^2}}{\sqrt{\sum_u\mu(u)f(u)^2}}.
\end{align*}
 In the above, we used Lemma~\ref{satzF} for the second inequality and the Cauchy-Schwarz inequality for the last one. Notice that
\begin{align}
&\sum_{u\sim v}w_{uv}|(|f(u)-g(u)|+|f(v)-g(v)|)^2\notag\\
\leq & \frac{1}{2}\sum_u\sum_{v,v\sim u}w_{uv}\left(2|f(u)-g(u)|^2+2|f(v)-g(v)|^2\right)\notag\\
\leq & 2d_{\mu}^w\Vert f-g\Vert_{\mu}^2.\label{NorNonnor2}
\end{align}
Inserting \eqref{NorNonnor2} into the above calculations we obtain
\begin{equation*}
 \beta^{\sigma}(V_F(\overline{t}), V_F(-\overline{t}))\leq 4k\mathcal{R}^{\sigma}(f)+4\sqrt{2}k\frac{\Vert f-g\Vert_{\mu}}{\Vert f\Vert_{\mu}}\sqrt{d_{\mu}^w\mathcal{R}^{\sigma}(f)}.
\end{equation*}
Noting that $f(u)\geq f(v)$ if and only if $F(u)\geq F(v)$ completes the proof.
\end{proof}

We are now ready to prove Theorem \ref{ImprCh}.

\begin{Lem}\label{lemma:oneoftwo}
For any non-zero function $f: V\rightarrow \mathbb{R}$ and any $1\leq k\leq n$, there exists $t'\in [0, \max_{u\in V}|f(u)|]$, such that at least one of the following estimates holds:
\begin{itemize}
  \item [(i)] $\beta^{\sigma}(V_f(t'), V_f(-t'))\leq 8k\mathcal{R}^{\sigma}(f)$;
  \item [(ii)] there exists $k$ disjointly supported functions $f_1, f_2, \ldots, f_k: V\rightarrow \mathbb{R}$ such that for each $1\leq i\leq k$,
\begin{equation}\label{eq:one of two}
 \mathcal{R}^{\sigma}(f_i)< 256d_{\mu}^wk^2\frac{\mathcal{R}^{\sigma}(f)^2}{\beta^{\sigma}(V_f(t'), V_f(-t'))^2}.
\end{equation}
\end{itemize}
\end{Lem}

\begin{proof}
Let $M:=\max_{u}|f(u)|$. We construct $2k+1$ real numbers
$t_0\leq t_1\leq \cdots \leq t_{2k}\leq M$ as follows: Set $t_0=0$. Suppose that we have already fixed $t_0, t_1,\ldots, t_{i-1}$. Now we look for $t_i\in [t_{i-1}, M]$ such that
\begin{align}
 \sum_{u: -t_i\leq f(u)<-t_{i-1}}&\mu(u) \, |f(u)-\psi_{-t_i,-t_{i-1}}(f(u))|^2\notag\\
&+ \sum_{u: t_{i-1}< f(u)\leq t_{i}} \mu(u) \, |f(u)-\psi_{t_{i-1},t_i}(f(u))|^2=C,\label{ProcedureGate}
\end{align}
where
$$C=\frac{\beta^{\sigma}(V_f(t'), V_f(-t'))^2\Vert f\Vert_{\mu}^2}{256k^3d_{\mu}^w\mathcal{R}^{\sigma}(f)}.$$
Notice that the left hand side of (\ref{ProcedureGate}) is continuous and non-decreasing with respect to\ $t_i$. We also point out again that
$\psi_{t_{i-1},t_i}(f(u))$ is the closest one of $\{t_{i-1},t_i\}$ to $f(u)$.
If we can find such constants satisfying (\ref{ProcedureGate}), we set the smallest one of them to be $t_i$; otherwise, we set $t_i=M$. This procedure is considered to be successful if $t_{2k}=M$.

If the procedure succeeds, we define a step function $g$ as in (\ref{DefOFg}). Then by definition,
\begin{equation}
 \Vert f-g\Vert_{\mu}^2\leq 2k  C.
\end{equation}
Applying Lemma \ref{preImprCheeger2}, we arrive at the inequality
\begin{equation*}
\beta^{\sigma}(V_f(t'), V_f(-t'))\leq 4k\mathcal{R}^{\sigma}(f)+\frac{\beta^{\sigma}(V_f(t'), V_f(-t'))}{2}.
\end{equation*}
Hence, the estimate (i) holds.

If, on the other hand, the procedure fails, that is, if we have $t_{2k}<M$, then we define $2k$ disjointly supported functions as
\begin{equation}
 f_i(u):=\left\{
          \begin{array}{cl}
            -|f(u)-\psi_{-t_{i-1},t_i}(f(u))|, & \hbox{if $-t_i\leq f(u)\leq -t_{i-1}$;} \\
            |f(u)-\psi_{t_{i-1}, t_i}(f(v))| & \hbox{if $t_{i-1}< f(u)\leq t_i$;} \\
            0 & \hbox{otherwise,}
          \end{array}
        \right.
\end{equation}
for $1\leq i\leq k$.
Recall that (\ref{ProcedureGate}) ensures $\Vert f_i\Vert^2=C$. Next we estimate the Rayleigh quotients of these functions.

\begin{Claim}\label{claim3}
 For any $\{u,v\}\in E$,
\begin{equation}
 \sum_{i=1}^{2k}|f_i(u)-\sigma(uv)f_i(v)|^2\leq |f(u)-\sigma(uv)f(v)|^2.
\end{equation}
\end{Claim}
Similarly to the proof of Lemma \ref{lemmaDisc1}, we only need to prove the claim when $\sigma(uv)=+1$.
\begin{description}
  \item[Case 1]  $u$ and $v$ lie in the support of the same function $f_i$. In this case,
$$\sum_{i=1}^{2k}|f_i(u)-f_i(v)|^2=|f_i(u)-f_i(v)|.$$
If $f(u)$ and $f(v)$ have the same sign, say $f(u)\geq 0$ and $f(v)\geq 0$, then we estimate
\begin{align*}
|f_i(u)-f_i(v)|^2&=\left||f(u)-\psi_{t_{i-1},t_i}(f(u))|-|f(v)-\psi_{t_{i-1},t_i}(f(v))|\right|^2\\
&\leq |f(u)-f(v)|^2.
\end{align*}
If $f(u)$ and $f(v)$ have different signs, say $f(u)>0$ and $f(v)<0$, then
\begin{align*}
|f_i(u)-f_i(v)|^2&=\left||f(u)-\psi_{t_{i-1},t_i}(f(u))|+|f(v)-\psi_{-t_{i},-t_{i-1}}(f(v))|\right|^2\\
&\leq \left||f(u)-\psi_{t_{i-1},t_i}(f(u))|+|-f(v)-\psi_{t_{i-1},t_{i}}(-f(v))|\right|^2\\
&\leq (|f(u)-t_{i-1}|+|-f(v)-t_{i-1}|)^2\leq |f(u)-f(v)|^2.
\end{align*}

\item [Case 2] $u\in\text{supp}(f_i), v\in \text{supp}(f_j)$, where $i\neq j$. We can assume $j>i$.
Then,
$$\sum_{i=1}^{2k}|f_i(u)-f_i(v)|^2=|f_i(u)|^2+|f_j(v)|^2.$$
If $f(u)$ and $f(v)$ have the same sign, say $f(u)\geq 0$ and $f(v)\geq 0$, then
\begin{align*}
 |f_i(u)|^2+|f_j(v)|^2&=|f(u)-\psi_{t_{i-1},t_i}(f(u))|^2+|f(v)-\psi_{t_{j-1},t_j}(f(v))|^2\\
&\leq |f(u)-t_i|^2+|f(v)-t_i|^2\leq |f(u)-f(v)|^2.
\end{align*}
If $f(u)$ and $f(v)$ have different signs, say $f(u)\geq 0$ and $f(v)\leq 0$, then we estimate
\begin{align*}
 |f_i(u)|^2+|f_j(v)|^2&=|f(u)-\psi_{t_{i-1},t_i}(f(u))|^2+|f(v)-\psi_{-t_{j},-t_{j-1}}(f(v))|^2\\
&=|f(u)-\psi_{t_{i-1},t_i}(f(u))|^2+|-f(v)-\psi_{t_{j-1},t_{j}}(-f(v))|^2\\
&\leq |f(u)-t_{i-1}|^2+|-f(v)-t_{j-1}|^2\leq |f(u)-f(v)|^2.
\end{align*}
\end{description}
This completes the proof of Claim.

Using the Claim, we calculate
\begin{align*}
 \sum_{i=1}^{2k}\mathcal{R}^{\sigma}(f_i)\sum_{i=1}^{2k}\mathcal{R}^{\sigma}(f_i)&=\frac{1}{C}\sum_{i=1}^{2k}\sum_{u\sim v}w_{uv}(f_i(u)-\sigma(uv)f_i(v))^2\\
&\leq \frac{1}{C}\sum_{u\sim v}w_{uv}|f(u)-\sigma(uv)f(v)|^2\\
&=256k^3d_{\mu}^w\frac{\mathcal{R}^{\sigma}(f)^2}{\beta^{\sigma}(V_f(t'), V_f(-t'))^2}.
\end{align*}
Let us abbreviate the above estimate as $ \sum_{i=1}^{2k}\mathcal{R}^{\sigma}(f_i)\leq \mathcal{C}k^3$, where
$$\mathcal{C}:=\frac{256d_{\mu}^w\mathcal{R}^\sigma(f)^2}{\beta^\sigma(V_f(t'),V_{f}(-t))^2}.$$
Then we can find $k$ functions from the set $\{f_1, f_2, \ldots, f_{2k}\}$, relabeling them as $f_1, f_2, \ldots, f_k$ if necessary, such that $\mathcal{R}^{\sigma}(f_i)< \mathcal{C}k^2$
 for $1\leq i\leq k$. This is true since otherwise there exist at least $k+1$ of $\{\mathcal{R}^\sigma(f_i): 1\leq i\leq n\}$ with the property $\mathcal{R}^\sigma(f)\geq\mathcal{C}k^2$ which leads to the contradiction that $\sum_{i=1}^{2k}\mathcal{R}^\sigma(f)\geq \mathcal{C}k^2(k+1)\geq \mathcal{C}k^3$.
That is, the estimate (ii) holds.
\end{proof}

\begin{proof}[Proof of Theorem \ref{ImprCh}]
Let $\phi_1$ be the eigenfunction of $\Delta^\sigma$ corresponding to $\lambda_1(\Delta^\sigma)$. We have $\lambda_1(\Delta^\sigma)=\mathcal{R}^\sigma(\phi_1)$. Moreover, by the definition (\ref{defSCheeger}) of $h_1^\sigma(\mu_d)$,
$$h_1^\sigma(\mu_d)\leq \beta^{\sigma}(V_{\phi_1}(t'), V_{\phi_1}(-t')), \,\,\,\text{for any}\,\,t'\in \left[0,\max_{u\in V}\phi_1(u)\right].$$
Combining the above two observations with Lemma \ref{lemma:oneoftwo}, we obtain the following property of $\phi_1$: For any $1\leq k\leq n$, at least one of the following two estimates holds:
\begin{equation*}
\text{(i) } h_1^{\sigma}(\mu_d)\leq 8k\lambda_1(\Delta^{\sigma}); \quad
\text{(ii) } \frac{\lambda_k(\Delta^\sigma)}{2}<256k^2\frac{\lambda_1(\Delta^\sigma)^2}{h_1^\sigma(\mu_d)^2}.
\end{equation*}
Note that in (ii), we have used Lemma \ref{lemmaPreliminary} and the fact that $d_{\mu_d}^w=1$ for the degree measure $\mu_d$. Now rearranging the estimate in (ii) and taking the square root of both sides, Theorem \ref{ImprCh} is proved.
\end{proof}

We further have the following higher-order estimates.
\begin{Thm}\label{ImprChIntroHiger}
There exists an absolute constant $C$ such that for any signed graph  $\Gamma=(G,\sigma)$ and any $1\leq k\leq l\leq n$,
\begin{equation}
 h_k^{\sigma}(\mu_d)<Clk^6\frac{\lambda_k(\Delta^{\sigma})}{\sqrt{\lambda_l(\Delta^{\sigma})}}.
\end{equation}
\end{Thm}
This generalizes the corresponding results for unsigned graphs given in \cite{KLLOT2013} and \cite{Liu13}.

\begin{proof}[Proof of Theorem \ref{ImprChIntroHiger}]
Combining Lemma \ref{lemma:localReyleighquotient} and the fact (\ref{eq:lambdakReyleigh}), we obtain the following property: For any $k \in \{1,2,\dots, n\}$, there exist $k$ disjointly supported functions $\psi_1, \psi_2, \ldots, \psi_k: V\rightarrow \mathbb{R}$ such that
\begin{equation}\label{eq:123}
\mathcal{R}^{\sigma}(\psi_i)\leq Ck^6\lambda_k(\Delta^\sigma),\,\,\,\text{for each}\,\, 1\leq i\leq k,
\end{equation}
where $C$ is an absolute constant.

Lemma \ref{lemma:oneoftwo} tells the following fact: For any $1\leq i\leq k$ and any $k\leq l\leq n$, there exists a $t_{i,l}\in [0,\max_{u\in V}|\psi_i(u)|]$, such that at least one of the following two estimates holds:
\begin{itemize}
  \item [(i)] $\beta^{\sigma}(V_{\psi_i}(t_{i,l}), V_{\psi_i}(-t_{i,l}))\leq 8l\mathcal{R}^{\sigma}(\psi_i)$;
  \item [(ii)]$\frac{\lambda_l(\Delta^\sigma)}{2}<256l^2\frac{\mathcal{R}^\sigma(\psi_i)^2}{\beta^{\sigma}(V_{\psi_i}(t_{i,l}), V_{\psi_i}(-t_{i,l}))^2}$.
\end{itemize}
Note that in (ii), we have used Lemma \ref{lemmaPreliminary} and the fact that $d_{\mu_d}^w=1$.

Let $i_0\in \{1,\ldots, k\}$ be the index satisfying
$$\beta^{\sigma}(V_{\psi_{i_0}}(t_{i_0,l}), V_{\psi_{i_0}}(-t_{i_0,l}))=\max_{1\leq i\leq k}\beta^{\sigma}(V_{\psi_i}(t_{i,l}), V_{\psi_i}(-t_{i,l})).$$
By the definition (\ref{eq:signedCheeger}) of $h_k^\sigma(\mu_d)$, we obtain
\begin{equation}\label{eq:345}
h_k^\sigma(\mu_d)\leq \beta^{\sigma}(V_{\psi_{i_0}}(t_{i_0,l}), V_{\psi_{i_0}}(-t_{i_0,l})).
\end{equation}
Inserting (\ref{eq:123}) and (\ref{eq:345}) into the estimates (i) and (ii), we obtain the following fact: For any $1\leq k\leq l\leq n$, at least one of the two estimates holds:
\begin{equation*}
\text{(i') } h_k^\sigma(\mu_d)\leq 8Clk^6\lambda_k(\Delta^\sigma); \quad
\text{(ii') } \frac{\lambda_l(\Delta^\sigma)}{2}<256l^2\frac{(Ck^6\lambda_k(\Delta^\sigma))^2}{h_k^\sigma(\mu_d)^2}.
\end{equation*}
Rearranging the estimate in (ii') and taking the square root of both sides, we obtain
$$h_k^{\sigma}(\mu_d)< 16\sqrt{2}Clk^{6}\frac{\lambda_k(\Delta^{\sigma})}{\sqrt{\lambda_l(\Delta^{\sigma})}}.$$
Using the fact $\lambda_l(\Delta^\sigma)\leq 2$, (i') implies
$$h_k^\sigma(\mu_d)\leq 8\sqrt{2}Clk^6\frac{\lambda_k(\Delta^\sigma)}{\sqrt{\lambda_l(\Delta^\sigma)}}.$$
This completes the proof.
\end{proof}
\begin{Rmk}
Theorems \ref{ImprChIntro} and \ref{ImprChIntroHiger} suggest that the well-known eigengap heuristic \cite{Luxburg07,KLLOT2013} for the traditional spectral clustering algorithm still holds for signed networks.
That is, in case that $\lambda_k(\Delta^{\sigma})$ is small and $\lambda_{k+1}(\Delta^{\sigma})$ is large, it is better to cluster the data into $k$ almost-balanced subgraphs.
\end{Rmk}

By setting $\mu=\mu_1$, we obtain the following results for $L^{\sigma}$.

\begin{Thm}\label{ImprChNonnormalized}
Given any signed graph $\Gamma=(G,\sigma)$ and any $1\leq k\leq n$, at least one of the following holds:
\begin{equation}
\text{(i). }h^{\sigma}_1(\mu_1)\leq 8k\lambda_1(L^{\sigma});\,\,\,\,\,\text{(ii). }h_1^{\sigma}(\mu_1)< 16\sqrt{2d_{\max}}k\frac{\lambda_1(L^{\sigma})}{\sqrt{\lambda_k(L^{\sigma})}}.
\end{equation}
\end{Thm}
Recalling that $\lambda_k(L^{\sigma})\leq 2d_{\max}$, we further obtain the following corollaries.

\begin{Cor}\label{CorImprChNonnormalized}
 For any signed graph $\Gamma$ and any $1\leq k\leq n$,
\begin{equation}
 h^{\sigma}_1(\mu_1)<16\sqrt{2d_{max}}k\frac{\lambda_1(L^{\sigma})}{\sqrt{\lambda_k(L^{\sigma})}}.
\end{equation}
\end{Cor}

\begin{Cor}\label{CorImprChNonHiger}
 There exists an absolute constant $C$ such that for any signed graph $\Gamma$ and $1\leq k\leq l\leq n$,
\begin{equation}
 h^{\sigma}_k(\mu_1)<C\sqrt{d_{max}}lk^6\frac{\lambda_k(L^{\sigma})}{\sqrt{\lambda_l(L^{\sigma})}}.
\end{equation}
\end{Cor}


 We comment at this point that the previous estimates about $h_k^\sigma(\mu_d)$ and $\lambda_k(\Delta^\sigma)$ can be directly translated into estimates for $\widetilde{h}_k^{\sigma}(\mu_d)$ and $2-\lambda_{n-k+1}(\Delta^{\sigma})$ by duality. This is due to Lemma \ref{lemma:dual}, which says that $$2-\lambda_{n-k+1}(\Delta^{\sigma})=\lambda_k(\Delta^{-\sigma}).$$ For example, the dual version of Theorem \ref{thmCheegerEsti} can be stated as follows.

\begin{Thm}\label{thmCheegerEstiDual}
Given a signed graph $\Gamma=(G, \sigma)$, we have
\begin{equation}\label{CheegerEstiDual}
\frac{2-\lambda_n(\Delta^{\sigma})}{2}\leq \widetilde{h}^{\sigma}_1(\mu_d)\leq \sqrt{2(2-\lambda_n(\Delta^{\sigma}))}.
\end{equation}
\end{Thm}

We omit the dual versions of Theorems~\ref{HigerCheeger}, \ref{ImprChIntro} and \ref{ImprChIntroHiger} here. Actually, these results are good demonstrations of a general antithetical duality principle discussed by Harary \cite{Harary57}.

\section{Signed triangles and the spectral gaps $\lambda_1$ and $2-\lambda_n$}
In this section, we prove an estimate for $\lambda_1(\Delta^\sigma)$ and $2-\lambda_n(\Delta^\sigma)$ in terms of the number of signed triangles. We will also discuss a similar result for the non-normalized Laplace matrix $L^\sigma$.


We first introduce some notation. For a given edge $\{u,v\}$, we divide the neighborhood $N_u:=\{u'| u'\sim u\}$ of a vertex $u$ into disjoint parts as
$$N_u=N_u^1\cup N^+_{uv} \cup N^-_{uv},$$
where
\begin{align*}
N_u^1&:=\{u'|u'\sim u, u'\not\sim v\},\\
N^+_{uv}&:=\{u'|u'\sim u, u'\sim v, \sigma(uv)\sigma(vu')\sigma(u'u)=+1\},\\
\text{and}\,\,\quad \quad \quad N^-_{uv}&:=\{u'|u'\sim u, u'\sim v, \sigma(uv)\sigma(vu')\sigma(u'u)=-1\}.
\end{align*}
Similarly, we have the partition $N_v=N_v^1\cup N^+_{uv} \cup N^-_{uv}$. Let us denote $$N_{uv}=N^+_{uv} \cup N^-_{uv},$$
and denote the number of positive and negative triangles including an edge $\{u,v\}$ by $\sharp^+(u,v)$ and $\sharp^-(u,v)$, respectively, where
$$\sharp^{+}(u,v):=\sum_{u'\in N^+_{uv}}1\quad \text{and}\quad\sharp^{-}(u,v):=\sum_{u'\in N^-_{uv}}1.$$

Note that the quantities $\sharp^+(u,v), \sharp^-(u,v)$ are switching invariant and their unsigned counterpart has an interesting close relation with the Ollivier-Ricci curvature of the underlying graph $G$ \cite{BJL,JostLiu14}. We prove the following theorem.

\begin{Thm}\label{thmBJLS} For a signed graph $\Gamma=(G,\sigma)$, \begin{equation}
\frac{w^2}{W}\frac{\min_{u\sim v}\sharp^-(u,v)}{\max_u d_u}\leq\lambda_1(\Delta^{\sigma})\leq\cdots\leq\lambda_N(\Delta^{\sigma})\leq 2-\frac{w^2}{W}\frac{\min_{u\sim v}\sharp^+(u,v)}{\max_u d_u},
\end{equation}
where $w=\min_{u\sim v}w_{uv}$ and $W=\max_{u\sim v}w_{uv}$.
\end{Thm}

This result is obtained by considering the iterated matrix $\Delta^{\sigma}[2]$ (see (\ref{iteratedmatrix}) below), extending an idea of Bauer, Jost and the second named author \cite{BJL} for the unsigned case.

\begin{proof}[Proof of Theorem \ref{thmBJLS}]
We consider an iterated matrix
\begin{equation}\label{iteratedmatrix}
\Delta^{\sigma}[2]=I-(D^{-1}A^{\sigma})^2.
\end{equation}
Then, for any function $f: V\rightarrow \mathbb{R}$ and any $u\in V$,
$$\Delta^{\sigma}[2]f(u)=f(u)-\frac{1}{d_u}\sum_v\,\sum_{u'\in N_{uv}}\frac{w_{u'u}w_{u'v}}{d_{u'}}\sigma(u'u)\sigma(u'v)f(v).$$
Let $f_n$ be the corresponding eigenfunction of $\lambda_n(\Delta^{\sigma})$. Then,
\begin{align}\label{fractionEigen}
\frac{(f_n,\Delta^{\sigma}[2]f_n)_{\mu}}{(f_n, \Delta^{\sigma} f_n)_{\mu}}=\frac{(f_n,[1-(1-\lambda_n(\Delta^{\sigma}))^2]f_n)_{\mu}}{(f_n, \lambda_n(\Delta^{\sigma})f_n)_{\mu}}=2-\lambda_n(\Delta^{\sigma}).
\end{align}
Note $\lambda_n(\Delta^{\sigma})(f_n,f_n)_{\mu}\neq 0$, hence the above expression is proper. Furthermore,
\begin{equation}\label{63}
(f_n, \Delta^{\sigma} f_n)_{\mu}=\sum_{u\sim v}(f_n(u)-\sigma(uv)f_n(v))^2,
\end{equation}
and
\begin{align*}
&(f_n, \Delta^{\sigma}[2]f_n)_{\mu}\\
=&\sum_u f_n(u)\sum_v\,\sum_{u'\in N_{uv}}\frac{w_{u'u}w_{u'v}}{d_{u'}}(f_n(u)-\sigma(u'u)\sigma(u'v)f_n(v))\\
=&\sum_{(u,v)}\,\sum_{u'\in N_{uv}}\frac{w_{u'u}w_{u'v}}{d_{u'}}(f_n(u)-\sigma(u'u)\sigma(u'v)f_n(v))^2\\
\geq &\sum_{u\sim v}\,\sum_{u'\in N^+_{uv}}\frac{w_{u'u}w_{u'v}}{d_{u'}}(f_n(u)-\sigma(u'u)\sigma(u'v)f_n(v))^2.
\end{align*}
In the above, $\sum_{(u,v)}$ stands for the summation over unordered pair of vertices $u,v$.
Inserting the above estimate and the equality (\ref{63}) into (\ref{fractionEigen}), we obtain
\begin{align*}
2-\lambda_n(\Delta^{\sigma})&\geq \frac{\sum_{u\sim v}\sum_{u'\in N^+_{uv}}\frac{w_{u'u}w_{u'v}}{d_{u'}}(f_n(u)-\sigma(u'u)\sigma(u'v)f_n(v))^2}{\sum_{u\sim v}w_{uv}(f_n(u)-\sigma(uv)f_n(v))^2}\\
&\geq \frac{w^2}{W}\min_{u\sim v}\sum_{u'\in N_{uv}^+}\frac{1}{d_{u'}}\geq \frac{w^2}{W} \frac{\min_{u\sim v}\sharp^+(u,v)}{\max_u d_u}.
\end{align*}
Using Lemma \ref{lemma:dual}, the lower bound estimate for $\lambda_1(\Delta^{\sigma})$ follows from duality.
\end{proof}

For the signed non-normalized Laplace matrix $L^{\sigma}$, we have the following estimate.

\begin{Thm}\label{DasS} For a signed unweighted graph $\Gamma=(G,\sigma)$,
\begin{equation}\lambda_N(L^{\sigma})\leq \max_{u\sim v}\{d_u+d_v-\sharp^+(u,v)\}.\end{equation}
\end{Thm}

This result
improves the estimate $\lambda_N(L^{\sigma})\leq \max_{u\sim v}\{d_u+d_v\}$ due to Hou, Li, and Pan \cite{HouLiPan03}.
In fact, Theorem \ref{DasS} answers the question asked in their paper \cite[remark after Theorem 3.5]{HouLiPan03}.

%
The techniques we used in the proof of Theorem \ref{thmBJLS} do not work for $L^{\sigma}$.
For example, we do not have a clear relation between the eigenvalues of $D-A^{\sigma}$ and $D-(A^{\sigma})^2$ as in (\ref{fractionEigen}) anymore. Actually, the underlying idea of the proof of Theorem \ref{thmBJLS} is that the normalized operator encodes certain random process on the graph, which is not true for the non-normalized operator. We will employ different ideas, which are adapted from Das \cite{Das} and Rojo \cite{Rojo}. In fact, we shall prove the following result.

\begin{Thm}\label{weightedDas}
For a signed graph $\Gamma=(G,\sigma)$,
\begin{align}
\lambda_n(L^{\sigma})\leq \frac{1}{2}\max_{u\sim v}\{&d_u+d_v+\sum_{u'\in N_u^1}w_{u'u}+\sum_{u'\in N_v^1}w_{u'v}\notag\\
&+\sum_{u'\in N_{uv}^-}(w_{u'u}+w_{u'v})
+\sum_{u'\in N_{uv}^+}|w_{u'u}-w_{u'v}|\}.\notag
\end{align}
\end{Thm}

Observe that Theorem \ref{DasS} is a direct corollary of this theorem, since when $\Gamma$ is a signed unweighted graph, we have
$d_u=|N_u^1|+|N_{uv}^+|+|N_{uv}^-|$.

\begin{proof}[Proof of Theorem~\ref{weightedDas}]
Let $f_n$ be the eigenfunction corresponding to $\lambda_n(L^{\sigma})$. Without loss of generality, suppose $f_n(u)=\max_{u'\in V}|f_n(u')|$, and
$$\sigma(uv)f_n(v)=\min_{u'\in N_u}\sigma(uu')f_n(u').$$
First, observe that $f_n(u)-\sigma(uv)f_n(v)\neq 0$. (Because otherwise we would have $\sigma(uu')f_n(u')=f_n(u)$ for any $u'\in N_u$,
which would imply $\lambda_n(L^{\sigma})f_n(u)=L^{\sigma}f_n(u)=0$,  a contradiction.)
Then we calculate
\begin{align*}
&\lambda_n(L^{\sigma})(f_n(u)-\sigma(uv)f_n(v))\\
& = L^{\sigma}f_n(u)-\sigma(uv)L^{\sigma}f_n(v)\\
& = d_{u}f_n(u)-\sum_{u'\in N_u}w_{u'u}\sigma(uu')f_n(u')\\
&-d_{v}\sigma(uv)f_n(v)
+\sum_{u'\in N_v}w_{u'v}\sigma(uv)\sigma(u'v)f_n(u').
\end{align*}
%
Using the facts that $\sigma(uu')f_n(u')\geq \sigma(uv)f_n(v)$ for any $u'\in N_u$ and $\sigma(uv)\sigma(u'v)\leq 1$ for any $u'\in N_v$, we continue to estimate:
\begin{align*}
&\lambda_n(L^{\sigma})(f_n(u)-\sigma(uv)f_n(v))\\
& \leq  d_{u}f_n(u)-\sigma(uv)f_n(v)\sum_{u'\in N_u^1\cup N_{uv}^-}w_{u'u}-d_{v}\sigma(uv)f_n(v)\\
&+f_n(u)\sum_{u'\in N_v^1\cup N_{uv}^-}
w_{u'v}+\sum_{u'\in N_{uv}^+}(w_{u'v}-w_{u'u})\sigma(u'u)f_n(u')\\
& = \frac{1}{2}\left(f_n(u)-\sigma(uv)f_n(v)\right)\left(d_u+d_v+\sum_{u'\in N_u^1\cup N_{uv}^-}w_{u'u}+\sum_{u'\in N_v^1\cup N_{uv}^-}w_{u'v}\right)
\\
&+\frac{1}{2}(f_n(u)+\sigma(uv)f_n(v))\left(d_u-\sum_{u'\in N_u^1\cup N_{uv}^-}w_{u'u}\right)
\\& -\frac{1}{2}(f_n(u)+\sigma(uv)f_n(v))\left(d_v-\sum_{u'\in N_v^1\cup N_{u,v}^-}w_{u'v}\right)\\
& +\sum_{u'\in N_{uv}^+}(w_{u'v}-w_{u'u})\sigma(u'u)f_n(u').
\end{align*}
Using the fact that $d_u-\sum_{u'\in N_u^1\cup N_{uv}^-}w_{u'u}=\sum_{u'\in N_{uv}^+}w_{u'u}$, we have
\begin{align}
&\lambda_n(L^{\sigma})(f_n(u)-\sigma(uv)f_n(v))\notag\\
\leq &\frac{1}{2}(f_n(u)-\sigma(uv)f_n(v))\left(d_u+d_v+\sum_{u'\in N_u^1\cup N_{uv}^-}w_{u'u}+\sum_{u'\in N_v^1\cup N_{uv}^-}w_{u'v}\right)\notag\\
&+\frac{1}{2}\sum_{u'\in N_{uv}^+}(w_{u'v}-w_{u'u})\left(f_n(u)+\sigma(uv)f_n(v)-2\sigma(u'u)f_n(u')\right).\label{finalkey}
\end{align}
For the latter term above, we further estimate
\begin{align*}
&\sum_{u'\in N_{uv}^+}(w_{u'v}-w_{u'u})\left(f_n(u)+\sigma(uv)f_n(v)-2\sigma(u'u)f_n(u')\right)\\
\leq &\sum_{u'\in N_{uv}^+}|w_{u'v}-w_{u'u}|(f_n(u)-\sigma(u'u)f_n(u'))\\
&+\sum_{u'\in N_{uv}^+}|w_{u'v}-w_{u'u}|(\sigma(u'u)f_n(u')-\sigma(uv)f_n(v))\\
=&(f_n(u)-\sigma(uv)f_n(v))\sum_{u'\in N_{uv}^+}|w_{u'v}-w_{u'u}|.
\end{align*}
Inserting the above estimate into (\ref{finalkey}), we arrive at
\begin{align*}
&\lambda_n(L^{\sigma})(f_n(u)-\sigma(uv)f_n(v))\\
\leq &\frac{1}{2}(f_n(u)-\sigma(uv)f_n(v))\left(d_u+d_v+\sum_{u'\in N_u^1\cup N_{uv}^-}w_{u'u}+\sum_{u'\in N_v^1\cup N_{uv}^-}w_{u'v}\right)\\
&+\frac{1}{2}(f_n(u)-\sigma(uv)f_n(v))\sum_{u'\in N_{uv}^+}|w_{u'v}-w_{u'u}|.
\end{align*}
This completes the proof.
\end{proof}

\section*{Acknowledgements}
The authors are very grateful to Norbert Peyerimhoff for suggesting the crucial idea leading to Proposition \ref{prop:switching invariant}. The authors thank Thomas Zaslavsky for his interest and valuable comments. The main part of this work was conceived when both authors were visiting the \emph{Zentrum f\"{u}r interdisziplin\"{a}re Forschung (ZiF)} of Bielefeld University. The authors thank the hospitality of ZiF and the financial support from the ZiF cooperation group {``Discrete and Continuous Models in the Theory of Networks"}. Finally, the authors acknowledge many useful comments of the anonymous referees.
SL was partially supported by the EPSRC Grant EP/K016687/1.
FMA acknowledges  the support of European Union's Seventh Framework
Programme (FP7/2007-2013) under grant agreement no.~318723 (MatheMACS).

\end{document}